\documentclass[11pt]{article}
\usepackage[margin=0.6in]{geometry}
\usepackage[numbers,sort&compress]{natbib}
\usepackage{graphicx}
\usepackage{amsmath,amsthm,bm,mathrsfs}
\renewenvironment{proof}[1][Proof]{\noindent\textit{#1. } }{\hfill$\square$}

 \newtheoremstyle{theorem}{6pt}{6pt}{\rm}{}{\sffamily}{ }{ }{}
 \theoremstyle{theorem}
\newtheorem{theorem}{\sc Theorem}[section]

 \newtheoremstyle{algorithm}{6pt}{6pt}{\rm}{}{\sffamily}{ }{ }{}
 \theoremstyle{algorithm}

 \newtheoremstyle{lemma}{6pt}{6pt}{\rm}{}{\sffamily}{ }{ }{}
 \theoremstyle{lemma}
 \newtheorem{lemma}{\sc Lemma}[section]

\newtheoremstyle{case}{6pt}{6pt}{\rm}{}{\sffamily}{. }{ }{}
 \theoremstyle{case}

 \newtheoremstyle{statement}{6pt}{6pt}{\rm}{}{\sffamily}{ }{ }{}
\theoremstyle{statement}

 \newtheoremstyle{corollary}{6pt}{6pt}{\rm}{}{\sffamily}{ }{ }{}
 \theoremstyle{corollary}
 \newtheorem{corollary}{\sc Corollary}[section]
 
  \newtheoremstyle{definition}{6pt}{6pt}{\rm}{}{\sffamily}{ }{ }{}
 \theoremstyle{definition}
 \newtheorem{definition}{\sc Definition}[section]

\newtheoremstyle{example}{6pt}{6pt}{\rm}{}{\sffamily}{ }{ }{}
\theoremstyle{example}
\newtheorem{example}[theorem]{\sc Example}

\newtheoremstyle{remark}{6pt}{6pt}{\rm}{}{\sffamily}{ }{ }{}
\theoremstyle{remark}
\newtheorem{remark}{\sc Remark}[section]

\newtheoremstyle{approximation}{6pt}{6pt}{\rm}{}{\sffamily}{ }{ }{}
\theoremstyle{approximation}

\newtheoremstyle{scheme}{6pt}{6pt}{\rm}{}{\sffamily}{ }{ }{}
\theoremstyle{scheme}

\newtheoremstyle{Algorithm}{6pt}{6pt}{\rm}{}{\sffamily}{ }{ }{}
\theoremstyle{Algorithm}

\newtheoremstyle{Assumption}{6pt}{6pt}{\rm}{}{\sffamily}{ }{ }{}
\theoremstyle{Assumption}

\newtheoremstyle{proposition}{6pt}{6pt}{\rm}{}{\sffamily}{ }{ }{}
\theoremstyle{proposition}
\newtheorem{proposition}{\sc Proposition}[section]

\newtheoremstyle{hypo}{6pt}{6pt}{\rm}{}{\sffamily}{ }{ }{}
 \theoremstyle{hypo}

  \newtheoremstyle{Step}{6pt}{6pt}{\rm}{}{}{ }{ }{}
 \theoremstyle{Step}

\usepackage{pifont}
\newcommand{\tick}{\ding{51}}%
\newcommand{\cross}{\ding{55}}%
\usepackage{listings}
\usepackage{xcolor,colortbl,framed}
\definecolor{Gray}{gray}{0.85}
\definecolor{LightCyan}{rgb}{0.88,1,1}
\newcolumntype{a}{>{\columncolor{Gray}}c}
\newcolumntype{b}{>{\columncolor{white}}c}
\usepackage{cleveref}

\numberwithin{equation}{section}

\newif\iftrversion \trversiontrue
\usepackage{sriram-macros}

\begin{document}

\title{Linear  Relaxations of Polynomial Positivity for Polynomial Lyapunov Function Synthesis.}
\author{ {\sc Mohamed Amin Ben Sassi${}^\dagger$, Sriram Sankaranarayanan${}^\dagger$, Xin Chen${}^*$ and Erika {\'A}brah{\'a}m${}^*$.}\\[2pt]
\begin{tabular}{l}
${}^\dagger$\ Department of Computer Science, University of Colorado, Boulder, CO, USA.\\[2pt]
${}^*$\ Department of Computer Science, RWTH Aachen University, Aachen, Germany.\\
\end{tabular}
}
\pagestyle{headings}
\markboth{Ben Sassi et al.}{\rm Linear Relaxation of Polynomial Positivity}
\maketitle


\begin{abstract}
{ We examine linear programming (LP) based relaxations for synthesizing polynomial Lyapunov functions to prove the  stability of polynomial ODEs. Our approach starts from a desired parametric polynomial form of the polynomial Lyapunov function. Subsequently, we encode the  positive-definiteness of the function, and the negation of its derivative, over the domain of interest. We first compare two classes of relaxations for encoding polynomial positivity: relaxations by sum-of-squares (SOS) programs, against relaxations based on Handelman representations and Bernstein polynomials, that produce linear programs.
  Next, we present a series of increasingly powerful LP relaxations based on expressing the given polynomial in its Bernstein form, as a linear combination of Bernstein polynomials. Subsequently, we show how these LP relaxations can be used to search for Lyapunov functions for polynomial ODEs by formulating LP instances.  We compare our techniques  with approaches based on SOS on a suite of automatically synthesized benchmarks. 
} {Positive Polynomials, Sum-Of-Squares, Bernstein Polynomials,
  Interval Arithmetic, Handelman Representations, Stability, Lyapunov
  Functions}
\end{abstract}

\section{Introduction}

The problem of discovering stability proofs for closed loop systems in
the form of Lyapunov functions, is an important step in the formal
verification of closed loop control
systems~\cite{Tabuada/09/Verification}. Furthermore, extensions of
Lyapunov functions such as \emph{control Lyapunov functions} can be
used to design controllers, and \emph{input-to-state stability (ISS)
  Lyapunov functions} are used to verify the stability of
inter-connected systems in a component-wise fashion.

In this paper, we focus on the synthesis of polynomial Lyapunov
functions for proving the stability of autonomous  systems
with polynomial dynamics using linear programming (LP) relaxations. At
its core, this requires us to find a positive definite polynomial
whose Lie derivatives are negative definite. Therefore, the problem of
finding a Lyapunov function depends intimately on techniques for
finding positive-definite polynomials over the domain $K$ of
interest. By finding a Lyapunov function over $K$ we ensure the
existence of a region (neighborhood of the equilibrium) contained in
$K$ such that the system is stable.  But proving that a multivariate
polynomial is positive definite over an interval is co-NP hard, and
therefore considered to a be a hard
problem~\cite{Garey+Johnson/1979/Computers}.  Many relaxations to this
problem have been studied, wherein a relaxed procedure can either
conclude that the polynomial is positive definite with certainty, or
fail with no conclusions. We examine two main flavors of relaxation:
\begin{enumerate}
\item The first class of \emph{linear representations} involve the
  expression of the target polynomial to be proven non-negative over
  the set $K$ of interest as a linear combination of polynomials that
  are known to be non-negative over the set $K$. This approaches
  reduces the polynomial positivity problem to a linear program (LP).
\item Alternatively, a different class of approaches uses ``Sum Of
  Squares representations"~\cite{Datta/2002/Computing}. This approach
  yields relaxations based on semi-definite programming
  (SDP)~\cite{Shor/1987/Class,Parillo/2003/Semidefinite,Lasserre/2001/Global}.
\end{enumerate}

As a first contribution of this paper, we extend the so-called
Handelman representations, considered in our previous
work~\cite{Sriram2013}, using the idea of Bernstein polynomials from
approximation
theory~\cite{Bernstein1915,Farouki/2012/Bernstein,Munoz+Narkavicz/2013/Formalization}. Bernstein
polynomials are a special basis of polynomials that have many rich
properties, especially over the unit interval $[0,1]$. For instance, tight
bounds on the values of these polynomials over the unit interval 
are known.  We show three LP relaxations, each more precise than the
previous, that exploit these bounds in the framework of a
\emph{reformulation linearization
  approach}~\cite{sherali91,sherali97}.  Next, we compare Bernstein
relaxations against SOS relaxations, demonstrating polynomials that
can be shown to be positive using one, but not the other.

Finally, the main contribution of the paper consists of adapting Bernstein
relaxations for finding Lyapunov functions over rectangular domain
$K$. The key difference is that, to find a Lyapunov function, we search
for a parametric polynomial $V(\vx,\vc)$ for unknown coefficients
$\vc$, which is positive definite, and whose derivative is negative
definite over the region of interest. A straightforward approach leads
to a bilinear program, that can be dualized  as a multi-parametric
program. We apply the basic requirements for a Lyapunov function, to
cast the multi-parametric program back into a LP, \emph{without any
  loss in precision}.

We have implemented the approach and describe our results on a suite
of examples. We also compare our work with a SOS programming
relaxation using the SOSTOOLS package~\cite{sostools}. On one hand, we
find that LP-based relaxations presented in this paper can find
Lyapunov functions for more benchmark instances while suffering from
fewer numerical issues when compared to a SOS programming
approach. Overall, the LP relaxations are shown to present a promising
approach for synthesizing Lyapunov functions over a bounded rectangle
$K$.

\subsubsection{Organization}

Section~\ref{sec:preliminaries} presents some preliminary notions of
Lyapunov functions, representations of positive polynomials including
Handelman, Schm{\"u}dgen and Putinar representations. We then present
the basic framework for synthesizing Lyapunov functions by formulating
a parametric polynomial that represents the desired
function. Section~\ref{Sec:lp-Bernstein-relaxations} presents the
basic properties of Bernstein polynomials and three LP relaxations for
proving polynomial positivity. In
Section~\ref{Sec:comparisons-relaxations}, we compare first Linear and
SOS relaxations then we compare the proposed Bernstein relaxations
with existing Linear ones.  Next, we describe the synthesis of
Lyapunov functions using Bernstein relaxations in
Section~\ref{Sec:poly-lyap-synth}.  Section~\ref{Sec:numerical-eval}
presents the numerical results.

\begin{framed}
An extended version of this paper including the  benchmark examples used in our evaluation along with the Lyapunov functions found for each is available
through arXiv~\cite{BenSassi+Others/2014/Synthesis}.
\end{framed}

\subsection{Related Work}

In this section, we restrict our discussion to those works that are
closely related to the overall problem of finding Lyapunov functions
for polynomial systems.

Much research has focused on the topic of stability analysis for
polynomial systems, which continues to be a challenging problem.  The
sum-of-squares (SOS) relaxation approach is quite popular, and has
been explored by many
authors~\cite{Papachristodoulou+Prajna/2002/Construction,jarvis2003,Tan+Packard/2006,Topcu+Packard+Seiler+Wheeler/2007/Stability}.
Papachristadoulou and Prajna were among the first to use SOS
relaxations for finding polynomial Lyapunov
functions~\cite{Papachristodoulou+Prajna/2002/Construction}. The core
idea is to express the polynomial and its negative Lie derivative as
sum-of-square polynomials for global stability analysis, or use a
suitable representation such as Putinar representation for finding
Lyapunov functions over a bounded region.  Their approach is
implemented in the SOSTOOLS package~\cite{sostools}. Extensions have
addressed the problem of controller synthesis~\cite{jarvis2003},
finding region of stability~\cite{Tan+Packard/2006}; and using a
combination of numerical simulations with SOS programming to estimate
the region of
stability~\cite{Topcu+Packard+Seiler+Wheeler/2007/Stability}.  A
related set of approaches directly relax the positivity of the
Lyapunov form and the negativity of its derivative using Linear Matrix
Inequalities
(LMIs)~\cite{Tibken/2000/Estimation,Chesi+Vicino/2005,Chesi/2007,Chesi/2009,
  Henrion+Lasserre/2006}.  Algebraic methods based, for example, on
Gr{\"o}bner basis~\cite{Forsman91constructionof}, or on constructive
semi-algebraic systems techniques have been
explored~\cite{She2009,She2013}.

While the approach in this paper focuses on polynomial system
stability using polynomial Lyapunov functions, the general problem of
analyzing nonlinear systems with rational, trigonometric and other
nonlinear terms has received lesser attention. Significantly,
Papachristadoulou et al.  present SOS relaxations for the stability of
non-polynomial systems through a process of \emph{algebraization} that
augments the original ODE with more state variables to create an
equivalent system involving rational
functions~\cite{Papachristodoulou+Prajna/2005/Analysis}.  Work by Chesi
addresses the use of LMI relaxations for the stability analysis of a
class of genetic regulatory networks involving ODEs with rational
functions on the right-hand sides~\cite{Chesi/2009}.

Conversely, polynomial systems often require non-polynomial Lyapunov functions.
Ahmadi et al. present an example of a polynomial system
that is globally stable but does not admit a polynomial Lyapunov
function~\cite{Parillo2011}. Some previous research, including
Papachristadoulou et
al. \emph{ibid.}~\cite{Papachristodoulou+Prajna/2005/Analysis}, has
focused on the generation of non-polynomial Lyapunov
functions. Recent work by Goubault et al. presents techniques for
finding rational, trigonometric and exponential Lyapunov functions for
polynomial systems through ideas from formal
integration~\cite{Goubault+Others/2014/Finding}. Their approach also
reduces to polynomial optimization problems, providing a future avenue
for the application of the linear relaxations developed here.

Recently, Ahmadi et al. have proposed different set of linear
relaxations for polynomial optimization problems called the DSOS
approach.  This approach further relaxes the positive-semidefiniteness
conditions in the SDP formulation using the condition of diagonal
dominance, that yields linear programming
relaxations~\cite{Ahmadi+Majumdar/2014/DSOS}. This idea has been also
been extended to synthesize polynomial Lyapunov
functions~\cite{Majumdar+Ahmadi+Tedrake/2014/Control}. A detailed
comparison of Ahmadi et al.'s ideas with those in this paper will be
carried out as part of our future work.

However, the use of LP relaxation has not received as much
attention. Johansen presented an approach based on
linear and quadratic programming~\cite{Johansen2000}.  This approach needs a so called
\emph{linear parametrization form} to reduce the stability conditions
to an infinite number of linear inequalities, which are reduced to a
finite number by discretizing the state space. As a consequence, the
number of linear inequalities characterizing the Lyapunov functions
grows exponentially with both the dimension of the state space and the
required accuracy.  Another approach using linear programming was
presented by Hafstein~\cite{Hafstein2002,Hafstein2014}. This approach
searches for a piecewise affine Lyapunov function, and requires a
triangulation of the state space.  Our approach derives polynomial (as
opposed to affine) Lyapunov function but also benefits from a
sub-division of the state-space to increase accuracy. The use of
Bernstein polynomial properties to formulate relaxations is a
distinguishing feature of our approach.  The recent work of Kamyar and
Peet, which remains under submission at the time of writing, also
examines linear relaxations for polynomial optimization problems using
Handelman representations, Bernstein polynomial representations (which
are closely related), and a linear relaxation based on the well-known
Polya's theorem for characterizing positive polynomials on a
simplex~\cite{Kamyar+Peet/2014/Polynomial}. As in this paper, they
have used their approach to search for Lyapunov functions by
decomposing the state space.  A key difference between the two papers
lies in our use of reformulation linearization that considers
nontrivial linear relationships between Bernstein polynomials. As
shown through examples in this paper, these relationships strictly
increase the set of polynomials that can be proven non-negative
through our linear relaxations. It must be mentioned that Kamyar et
al.  consider more applications including searching for piecewise
polynomial Lyapunov functions and the robust $H_{\infty}$ control of
systems. Our future work will consider the application of the LP
relaxations to those considered in Kamyar et al, facilitating an
experimental comparison. Ratschan and She use interval arithmetic
relaxations with branch-and-bound to discover Lyapunov like functions
to prove a notion of region stability of polynomial
systems~\cite{Ratschan+She/2010/Providing}.  This is extended in our
previous work to find LP relaxations using the notion of Handelman
representations~\cite{Sriram2013}. In practice, the interval
arithmetic approach is known to be quite coarse for proving polynomial
positivity, especially for intervals that contain $0$. Therefore,
Ratschan and She restrict themselves to region stability by excluding
a small interval containing the equilibrium from their region of
interest. Furthermore, the coarseness of interval relaxation is
remedied by resorting to branch-and-bound over the domain. A detailed
comparison between interval and Handelman approach is provided in our
previous work~\cite{Sriram2013}, wherein we conclude that both
approaches have complementary strengths. A combined approach is thus
formulated.

In this paper, we start from such a combined approach and
generalize it further through Bernstein polynomials. We use
non-trivial properties of Bernstein polynomials that cannot be proven
through interval analysis or Handelman representations, to further
improve the quality of these relaxations.
Section~\ref{Sec:comparisons-relaxations} provides detailed
comparisons between the various approaches presented in this paper
with the approaches based on interval arithmetic, Handelman
representations and SOS programming relaxations.

\section{Preliminaries}\label{sec:preliminaries}
In this section, we recall the definition of Lyapunov functions and
discuss procedures for synthesizing them. Subsequently, we examine two
techniques for proving the positivity of polynomials:
so-called \emph{Handelman representation} technique that produces
linear programming (LP) relaxations and a \emph{Putinar representation}
technique that produces semi-definite programming (SDP) relaxations.
We extend these to recall algorithmic schemes for synthesizing
Lyapunov functions, wherein we treat constraints that arise from the
positivity of polynomials parameterized by unknown coefficients.

\begin{definition}[Positive Semi-Definite Functions]
A function $f: \reals^n \rightarrow \reals$ is \emph{positive
  semi-definite} over a domain $U \subseteq \reals^n$ iff 
\[ (\forall\ \vx \in U)\ f(\vx) \geq 0 \,.\]
 Furthermore, $f$ is \emph{positive definite} iff $f$ is positive
 semi-definite, and additionally, (a) $f(\vx) > 0$ for all $\vx \in U
\setminus \{ 0 \}$, and (b) $f(0) = 0$.
\end{definition}

\subsection{Lyapunov Functions}
We now recall the key concepts of stability and Lyapunov functions.
Let $\scr{S}$ be a continuous system over a state-space
$\scr{X} \subseteq \reals^n$ specified by a system of ODEs
\[ \frac{d\vx}{dt} = f(\vx),\ \vx \in \scr{X} \,.\] 

We assume that the right-hand side function $f(\vx)$ is Lipschitz
 continuous over $\vx$.  An equilibrium of the system
 $\vx^* \in \scr{X}$ satisfies $f(\vx^*) = 0$.

\begin{definition}[Lyapunov and Asymptotic Stability]
A system is \emph{Lyapunov stable} over an open region $U$ around the
equilibrium $\vx^*$, if for every neighborhood $N \subseteq U$ of
$\vx^*$ there is a neighborhood $M \subset N$ such that $ (\forall\
\vx(0) \in M)\ (\forall\ t \geq 0)\ \vx(t) \in N$.  A system is
\emph{asymptotically} stable if it is Lyapunov stable and all
trajectories starting from $U$ approach $\vx^*$ as $t \rightarrow \infty$. 
\end{definition}
Lyapunov functions are useful in proving that a system is
stable in a region around the equilibrium.  Without loss of
generality, we assume that $\vx^* = \vec{0}$. The definitions below
are based on the terminology used by Meiss (\cite{Meiss/2007/Differential}).
\begin{definition}\label{Def:lyapunov-function}
  A continuous and differentiable function $V(\vx)$ is a \emph{weak
    Lyapunov function} over a region $U \subseteq \scr{X}$ iff the
  following conditions hold:
\begin{enumerate}
\item $V(\vx)$ is positive definite over $U$, i.e, $V(\vx) > 0 $ for
  all $\vx \in U \setminus \{ \vec{0} \}$ and $V(\vec{0}) = 0$.
\item $\frac{dV}{dt} = (\grad V \cdot f) \leq 0$ for all $\vx \in U$.
\end{enumerate}
Additionally, $V$ is a \emph{strong Lyapunov function} if
$ \left(- \frac{dV}{dt}\right)$ is positive definite.
\end{definition}
Weak Lyapunov functions are used to prove that a system is Lyapunov
stable over a  subset of region $U$, whereas a strong Lyapunov
function proves asymptotic stability. The approaches presented in this
paper can be used to search for weak as well as strong Lyapunov
functions.

Stability is an important property of control systems. Techniques for
discovering Lyapunov functions to certify the stability of a closed
loop model are quite useful in control systems design.

\subsection{Proving Polynomial Positivity}\label{Sec:positivity-representations}

At the heart of Lyapunov function synthesis, we face the challenge of
establishing that a given function $V(\vx)$ is positive (negative)
definite over $U$.  The problem of deciding whether a given polynomial
$V(\vx)$ is positive definite is
NP-hard~\cite{Garey+Johnson/1979/Computers}.  A precise solution
requires a decision procedure over the theory of reals.
~\cite{Tarski/51/Decision,Collins/75/Quantifier}.  To wit, we check
the validity of the formula: $ (\forall\ \vx
\in U)\ V(\vx) \geq 0 $ using tools such as QEPCAD
~\cite{Collins+Hong/91/Partial} and REDLOG
~\cite{Dolzmann+Sturm/97/REDLOG}.  This process is \emph{exact}, but
intractable for all but the smallest of systems and low degree
polynomials for $V$.  Therefore, we seek stricter versions of positive
semi-definiteness that yield a more tractable system of constraints.

We examine relaxations to the problem of establishing that a given
polynomial is positive semi-definite over a region
$K \subseteq \reals^n$.  In the literature, we can distinguish two
kind of techniques for establishing that a given polynomial is
positive semi-definite~\cite{Powers2000}.  Here, we call them 
\emph{linear representations} and \emph{sum of square (SOS) representations}.

\subsubsection{Linear Representations} 
The first approach writes the given polynomial $p$ as a conic
combination of products of the constraints defining $K$. This idea was
first examined by Bernstein for proving the positivity of univariate
polynomials over the unit interval $[0,1]$~\cite{Bernstein1915}.
Furthermore, Hausdorff~\cite{Hausdorff1921} extended it to the interval
$[-1,1]$.
\begin{theorem}(Bernstein and Hausdorff).
\label{theo:bern}
A polynomial $p(x)$ is strictly positive over $[-1,1]$ iff there exists a degree $d > 0$
and exists non-negative constants $\lambda_{0},\ldots, \lambda_d \geq 0$, such that 
\begin{equation}
\label{eq:bernhaus}
 p(x) \equiv \displaystyle{\sum_{i=0}^d \lambda_{i}(1-x)^i(1+x)^{d-i}},
\end{equation}

\end{theorem}

This approach is generalized to  multivariate polynomials
over $\vx:(x_1,\ldots,x_n)$ and general semi-algebraic sets $K \subseteq \reals^n$
rather than the unit interval.  Let $K$ be defined as a semi-algebraic set:
\[ K: (p_1(\vx) \geq 0\ \land\ \cdots \ \land\ p_m(\vx) \geq 0) \]
for multivariate polynomials $p_1,\ldots,p_m$.  A power-product over
the set of polynomials $P:\ \{ p_1, \ldots,p_m\}$ is a polynomial of
the form $f:\ p_1^{n_1}p_2^{n_2} \cdots p_m^{n_m}$. The degree of the
power-product is given by $(n_1,\ldots,n_m)$. We say that
$(n_1,\ldots,n_m) \leq D$ iff $n_j \leq D$ for each $j \in [1,m]$.
Let $\mathsc{pp}(P,D)$ represent all power products from the set $P$
bounded by degree $D$.

\begin{theorem}[Conic Combination of Power Products]\label{Theorem:conic-combination-of-power-products}
If a polynomial $p$ can be written as a conic combination of power-products of 
$P:\ \{p_1,\ldots,p_m\}$, i.e,
\begin{equation}\label{Eq:handelman-rep}
 p(\vx)\ \equiv\ \sum_{f \in \mathsc{pp}(P,D)}\ \lambda_f f,\ \; \mbox{s.t.}\ (\forall\ f \in \mathsc{pp}(P,D))\ \lambda_f \geq 0 \,,
\end{equation}
then the polynomial $p$ is non-negative over
$K$:
\[ (\forall\ \vx \in \reals^n)\ \vx \in K\ \Rightarrow\ p(\vx) \geq 0\,. \]
\end{theorem}
The proof is quite simple. The conic combination of power-products in
$\mathsc{pp}(P,D)$ as shown in Eq.~\eqref{Eq:handelman-rep}, is said
to be a \emph{Handelman representation} for a polynomial $p$~\cite{Datta/2002/Computing}.
However, the converse of
Theorem~\ref{Theorem:conic-combination-of-power-products} does not
hold, in general. Therefore, polynomials that are positive
semi-definite over $K$ need not necessarily have a Handelman
representation.
\begin{example}
Consider the first orthant in $\reals^2$ given by $K_1:\ (x_1 \geq
0\ \land\ x_2 \geq 0 )$ and the polynomial $p: x_1^2 - 2 x_1 x_2 +
x_2^2$. It is easily seen that $p$ cannot be written as a conic
combination of power products over $x_1, x_2$, no matter what the
degree limit $D$ is chosen to be.

\end{example}
An important question is when the converse of
Theorem~\ref{Theorem:conic-combination-of-power-products} holds.  One
important case for a compact, polyhedron $K$ defined as
$K:\displaystyle{\bigwedge\limits_{j=1}^m \underset{f_j}{\underbrace{(\va_j \vx- \vec{b_j})}}\ge 0}$ is given by Handelman~\cite{Handelman/1988/Representing}.
 Let $P$ denote the set  $\{ f_1,\ldots,f_m\}$, and $\mathsc{pp}(P,D)$ denote the power products of degree up to $D$, as before.
\begin{theorem}[Handelman]
\label{theo:handel}
If $p$ is strictly positive over a compact polyhedron $K$ then there
exists a degree bound $D > 0$ such that
\begin{equation}
\label{eq:handel}
 p \equiv \displaystyle{\sum_{f \in \mathsc{pp}(P,D)}\ \lambda_{f} f}, \text{ for }\lambda_{f}\ge 0 \,.
\end{equation}
\end{theorem}

\begin{example}\label{Ex:handelman-1}
Consider a polynomial $p(x_1,x_2) = -2 x_1^3 + 6 x_1^2x_2 + 7x_1^2 - 6 x_1 x_2^2 - 14 x_1 x_2 + 2 x_2^3 + 7 x_2^2 - 9$
over the set $K: ( \underset{f_1}{\underbrace{x_1 - x_2 -3}} \geq 0\ \land\ \underset{f_2}{\underbrace{x_2 - x_1 -1 }} \geq 0 )$. We can establish
the positivity of $p$ over $K$ through its Handelman representation:
\[ p \equiv 2 f_1^2 f_2 + 3 f_1 f_2 \]
\end{example}

The problem of checking if a polynomial $p$ is positive semi-definite over a set $K: \bigwedge\limits_{j=1}^m p_j(\vx) \geq 0$ is therefore tacked as follows:
\begin{enumerate}
\item Choose a degree limit $D$ and construct all terms in $\mathsc{pp}(P,D)$, where $P = \{ p_1, \ldots,p_m\}$ are the polynomials defining $K$.
\item Express $p \equiv \sum\limits_{f \in \mathsc{pp}(P,D)}\ \lambda_f f $ for \emph{unknown multipliers} $\lambda_f \geq 0$.
\item Equate coefficients on both sides (the given polynomial and the Handelman representation) to obtain a set of linear inequality constraints involving $\lambda_f$. 
\item Use a Linear Programming (LP) solver to solve these constraints. If feasible, the result yields a proof that $p$ is positive semi-definite over $K$.
\end{enumerate}

We note that the procedure fails if $p$ is not positive-definite over
 $K$, or $p$ does not have a Handelman representation over
 $K$. Nevertheless, it provides an useful LP relaxation for polynomial
 positivity.

\subsubsection{Sum-Of-Squares representations}
Another important approach to proving positivity is through the
well-known sum-of-squares (SOS) decomposition.

\begin{definition}
A polynomial $p(\vx)$ is a sum-of-squares (SOS) iff there exists
polynomials $p_1,\ldots,p_k$ over $\vx$ such that $p$ can be written
as
\[ p \equiv p_1^2 + \ldots + p_k^2 \]
\end{definition}

It is easy to show that any SOS polynomial is positive semi-definite
over $\reals^n$. On the other hand, not every positive semi-definite
polynomial is SOS (the so-called Motzkin polynomial provides a
counter-example)~\cite{Motzkin}.

\paragraph{Schm{\"u}dgen Representation: }\label{sec:schmudgen-represenatation}
Whereas SOS polynomials are positive semidefinite over $\reals^n$, we
often seek if $p$ is positive semi-definite over a semi-algebraic set
$K: \ ( p_1 \geq 0\ \land\ \cdots \ \land\ p_m \geq 0)$.

We define the \emph{pre-order} generated by a set
$P=\{p_1,\ldots,p_m\}$ of polynomials as the set 
\[ R(P)=\{ p_1^{e_1}p_2^{e_2}\cdots p_m^{e_m}\ |\ (e_1,\ldots,e_m) \in \{
0,1\}^m \} \,.\]
It is easy to see that if for some given $\vx$, $ p_i(\vx) \geq 0$ for all $i \in [1,m]$, then
for each $r \in R(P)$, we have $r(\vx) \geq 0$. In fact, the following result follows
easily:

\begin{theorem}\label{Theorem:schmudgen-represenation-trivial}
If a polynomial $p$ can be expressed as \emph{SOS polynomial combination} of elements in $R(P)$, 
\begin{equation}\label{Eq:schmudgen-representation}
p \equiv \sum_{r \in R(P)} q_r r \,\ \mbox{for SOS polynomials}\ q_r \,,
\end{equation}
then $p$ is positive semi-definite over $K$.
\end{theorem}
Decomposing a polynomial $p$ according to
~\cref{Eq:schmudgen-representation} will be called
the \emph{Schm{\"u}dgen representation} of $p$. The terminology is
inspired by the following result due to
Schm{\"u}dgen~\cite{Schmudgen1991}:

\begin{theorem}[Schm{\"u}dgen Positivstellensatz]
If $K$ is compact then every polynomial $p(\vx)$ that is strictly
positive over $K$ has a Schm{\"u}dgen representation of the form given
in ~\cref{Eq:schmudgen-representation}.
\end{theorem}

While Schm{\"u}dgen representations are powerful, and in fact, subsume
the Handelman representation approach, or even the Bernstein
polynomial relaxations to be presented in
Section~\ref{Sec:lp-Bernstein-relaxations}, the computational cost of
using them is prohibitive. Using the
form~\cref{Eq:schmudgen-representation} requires finding $2^m$ SOS
polynomials. In our applications, $K$ typically represents the unit
rectangle $[-1,1]^n$, which makes the size of a Schm{\"u}dgen
representation exponential in the size of the variables.  As a result,
we will not consider this representation any further in this paper.

\paragraph{Putinar Representation: }\label{sec:putinar-representation}
The Putinar representation approach provides a less expensive
alternative.  Once again, let $K: ( p_1 \geq 0\ \land\ \cdots \ \land\
p_m \geq 0)$ be a set of interest.

\begin{theorem}\label{Theorem:putinar-representation-trivial}
If a polynomial $p$ can be expressed as 
\begin{equation}\label{Eq:putinar-representation}
p \equiv q_0 + q_1 p_1 + \cdots + q_m p_m 
\end{equation}
for SOS polynomials $q_0, \ldots,q_m$, then $p$ is positive
semi-definite over $K$.
\end{theorem}
Decomposing a polynomial $p$ according to
Equation~\ref{Eq:putinar-representation} is said to provide
a \emph{Putinar representation} for $p$.  The converse of Theorem~\ref{Theorem:putinar-representation-trivial}
was proved by Putinar~\cite{Putinar1993}.
\begin{theorem}(Putinar)
\label{theo:putin}
 Let $K:( p_1 \geq 0\ \land\ \cdots \ \land\ p_m \geq 0) $ be a
 compact set, and suppose there exists a polynomial $p_0$ of the form
 $p_0 = r_0 + \sum_{j=1}^m r_i p_i$ where $r_0,\ldots,r_m$ are all
 SOS, and the set $\hat{K}:\ \{ \vx\ |\ p_0(\vx) \geq 0 \}$ is also
 compact.

It follows that every polynomial $p(\vx)$ that is strictly positive on $K$ has
a Putinar representation: $ p \equiv
q_0+\displaystyle{\sum_{j=1}^{m}q_j p_j}$ for SOS polynomials
$q_0,\ldots,q_m$.
\end{theorem}

A Putinar representation of $p$  for  a set $P = \{ p_1,\ldots,p_m\}$ involves expressing
$p \ \equiv\ q_0 + \sum_{j=1}^m q_j p_j$ for a SOS polynomial $q_j$.  Searching whether
a polynomial $p$ is positive semi-definite over $K:\ \bigwedge\limits_{p_j \in P} p_j \geq 0$
involves searching for a Putinar representation. 
\[ \mbox{find}\ q_0,\ldots,q_m\ \mbox{s.t.}\ p \equiv  q_0 + \sum_{j=1}^m q_j p_j,\ q_0,\ldots,q_m\ \mbox{are SOS} \,.\]
The key steps involve parameterizing $q_0,\ldots,q_m$ in terms of
polynomials of bounded degree $D$ over a set of unknown coefficients
$\vc$, and then solving the resulting problem through a relaxation to
semi-definite programming, originally proposed by Shor and further
developed by Parillo~\cite{Shor/1987/Class,Parillo/2003/Semidefinite}.
The resulting optimization problem is called a Sum-of-Squares
programming problem (SOS).

\subsection{Synthesis of Lyapunov Functions}\label{Sec:Lyapunov-synth}

We now summarize the standard approach to synthesizing Lyapunov
functions using Handelman or Putinar representations.  The Handelman
approach reduces the synthesis to solving a set of linear programs,
and was presented in our previous
work~\cite{Sriram2013}. The Putinar
representation approach uses SOS programming, and was presented by
Papachristadoulou et
al.~\cite{Papachristodoulou+Prajna/2002/Construction}. This approach
is implemented in a package SOSTOOLS that provides a user-friendly
interface for posing SOS programming problems and solving them by
relaxing to a semi-definite program~\cite{sostools}.

Let $U \subseteq \reals^n$ be a compact set and $\scr{S}$ be a system
defined by the ODE $\frac{d\vx}{dt} = f(\vx)$. We assume that the
origin is the equilibrium of $\scr{S}$, i.e, $f(\vec{0}) = 0$,
$\vec{0} \in \mathsf{interior}(U)$, and wish to prove local asymptotic
(or Lyapunov) stability of $\scr{S}$ for a subset of the region $U$.

Therefore, we seek a Lyapunov function of the form $V(\vx, \vc)$, 
wherein $V$ is a polynomial form over $\vx$ whose coefficients are polynomials over
$\vc$. Let $V'$ denote the Lie derivative of $V$, i.e, 
$V'(\vx,\vc) = (\nabla_{\vx}\ V)\cdot f$. 
We define the set $C$ as follows:
\begin{equation}\label{Eq:def-c}
  C = \{ \vc \ |\ V(\vx,\vc) \ \mbox{is\ positive definite for}\ \vx \in U \} \,.
\end{equation}
Also, let $\hat{C}$ represent the set:
\begin{equation}\label{Eq:def-c-prime}
\hat{C} = \{ \vc\ |\ V'(\vx,\vc)\ \mbox{is negative definite for}\ \vx \in U \}\,.
\end{equation}
We replace negative definiteness for negative semi-definiteness if
Lyapunov stability, rather than asymptotic stability is of interest.
The overall procedure for synthesizing Lyapunov functions proceeds as
follows:
\begin{enumerate}
\item Fix a template form $V(\vx,\vc)$ with parameters $\vc$.
\item Compute constraints $\psi[\vc]$ whose solutions yield the set $C$ in 
Equation~\eqref{Eq:def-c}.
\item Compute constraints $\hat{\psi}[\vc]$ whose solutions yield the set $\hat{C}$ from
Equation~\eqref{Eq:def-c-prime}.
\item Compute a value $\vc \in C \cap \hat{C}$ by solving the constraints
  $\psi \land \hat{\psi}$. The resulting function $V_c(\vx)$ is a Lyapunov
  function.
\end{enumerate}

The main problem, therefore, is to characterize a set $C$ for the
unknown parameters $\vc$, so $V_c(\vx)$ is positive definite over $U$
for all $\vc \in {C}$. Thus, the process of searching for Lyapunov
functions of a given form devolves into the problem of finding a
system of constraints for the sets ${C}, \hat{C}$.

\begin{remark}
  It must be remarked that finding a (strong) Lyapunov function
  $V(\vx)$ inside a region $U$, as presented thus far, does not
  necessarily prove that the system is asymptotically stable for every
  initial state $\vx \in U$. For instance, trajectories starting from
  $\vx \in U$ may exit the set $U$. 

  However, let $\gamma$ represent
  the largest value such that for all $\vx \in U$, $V(\vx) \leq
  \gamma$.
  \[ \gamma:\ \max_{\vx \in U}\ V(\vx) \]
It can be shown that the system is 
asymptotically stable inside the  set $V_{\gamma}:\ \{ \vx | V(\vx) \leq \gamma \}$.
\end{remark}

Handelman representations and Putinar representations provide us two
approaches to encoding the positive definiteness of $V$ and negative
definiteness of $V'$ to characterize the sets $C,\hat{C}$.

\paragraph{Handelman Representations:} We now briefly summarize our previous
work that uses Handelman representations for Lyapunov function
synthesis~\cite{Sriram2013}.

 Let us assume that the set $U$ is written as a semi-algebraic set:
\[ U: \bigwedge\limits_{j=1}^m p_j(\vx) \geq 0 \]
Let $P = \{ p_1, \ldots,p_m\}$ represent these constraints.  Given a
degree limit $D$, we construct the set $\mathsc{pp}(P,D)$ of all
power-products of the form $\prod\limits_{j=1}^m p_j^{n_j}$ wherein
$0 \leq n_j \leq D$.

We encode positive semi-definiteness of a form $V(\vx,\vc)$ by writing it as
\begin{equation}\label{Eq:handelman-repr-eq}
 V(\vx,\vc) \equiv \sum\limits_{f \in \mathsc{pp}(P,D)}\ \lambda_f f\,\ \mbox{wherein}\ \lambda_f \geq 0 \,.
\end{equation}
Positive definiteness is encoded using a standard trick presented by
Papachristodoulou et
al.~\cite{Papachristodoulou+Prajna/2002/Construction}.  Briefly, the
idea is to write $V= \hat{V} + \sum_{j=1}^n \epsilon x_j^{2p}$ for
$\hat{V}(\vx,\vc)$, an unknown positive semi-definite function and a
fixed positive definite contribution given by setting
$\epsilon,p$. This idea is used in all our examples wherein positive
definiteness is to be encoded rather than positive semi-definiteness.

We eliminate $\vx$ by equating the coefficients of monomials on both
sides of~\cref{Eq:handelman-repr-eq}, and obtain a set of linear
constraints $\psi[\vc,\vec{\lambda}]$ involving $\vc$ and
$\vec{\lambda}$.  The set $C$ is characterized as a polyhedron
obtained by the projection
\[ C:\ \{ \vc\ |\ \exists\ \vec{\lambda} \geq 0\ \psi(\vc,\vec{\lambda}) \} \,.\]
In practice, we do not project $\vec{\lambda}$, but instead retain $\psi$ as a set
of constraints involving both $\vc,\vec{\lambda}$. Similarly, we consider the
Lie derivative $V'(\vc,\vx)$ and obtain constraints $\psi(\vc,\vec{\mu})$
for a different set of multipliers $\vec{\mu}$. 

The overall problem reduces to finding a value of $\vec{c}$ that satisfies the
constraints
\[ \psi(\vc,\vec{\lambda})\ \land\ \hat{\psi}(\vc,\vec{\mu}) \,,\ \]
for some $\vec{\lambda}, \vec{\mu} \geq 0$. This is achieved by solving
a set of linear programs.

\begin{example}
  Consider a parametric polynomial $p(\vc,\vx): c_1 x_1^2 + c_2 x_2^2
  + c_3 x_1 x_2 + c_4 x_1 + c_5 x_2 + c_6$ and the set $K$ defined by
  the constraints: $ x_2 - x_1 \geq -1\ \land\ x_1 + x_2 \geq 2 $. We
  will use a Handelman representation to characterize a set of
  parameters $C$ s.t.  $ \vx \in K \models
  p(\vc,\vx) \geq 0$. Using degree-2 Handelman representation, we obtain the following larger set
of constraints:
\[ \begin{array}{cl}
x_2 - x_1 + 1 \geq 0\ \land\ x_1 + x_2 - 2 \geq 0\ \land\\
  x_2^2 + x_1^2 - 2 x_1 x_2 + 2 x_2 - 2 x_1 + 1 \geq 0 \land & \leftarrow\ (x_2 - x_1 + 1)^2 \geq 0 \\
  x_1^2 + x_2^2 +2 x_1 x_2 - 4 x_1  - 4 x_2 + 4 \geq 0\ \land\  &  \leftarrow\ (x_1 + x_2 -2)^2 \geq 0 \\
-x_1^2 + x_2^2 + 3 x_1 - x_2 -2 \geq 0 & \leftarrow (x_2 - x_1 + 1) (x_1 + x_2 -2) \geq 0 \\
\end{array}\]

We express $p$ as a linear combination of these constraints yielding the following equivalence:
\[ c_1 x_1^2 + c_2 x_2^2
  + c_3 x_1 x_2 + c_4 x_1 + c_5 x_2 + c_6 \ \equiv\ \left(\begin{array}{c} 
\lambda + \lambda_0 ( x_2 - x_1 + 1) + \lambda_1 ( x_1+ x_2 -2) + \\
\lambda_2 (  x_2^2 + x_1^2 - 2 x_1 x_2 + 2 x_2 - 2 x_1 + 1 ) + \\
\lambda_3( x_1^2 + x_2^2 +2 x_1 x_2 - 4 x_1  - 4 x_2 + 4  ) + \\
\lambda_4( -x_1^2 + x_2^2 + 3 x_1 - x_2 -2) 
\end{array} \right)\]
where $\lambda,\lambda_0, \ldots, \lambda_4 \geq 0$.
Matching coefficients of monomials on both sides, we obtain linear inequality constraints
involving variables $c_1, \ldots, c_6$ and $\lambda_1, \ldots, \lambda_4$:
\[ \begin{array}{rcll}
c_1 & = & \lambda_2 + \lambda_3 - \lambda_4  & \leftarrow\ \mbox{Matching}\ x_1^2\\
c_2 &=& \lambda_2 + \lambda_3 +\lambda_4 & \leftarrow\ \mbox{Matching}\ x_2^2\\
c_3 &=& -2 \lambda_2 + 2 \lambda_3 & \leftarrow\ \mbox{Matching}\ x_1 x_2 \\
& \vdots &  & \leftarrow\ \mbox{Matching}\ x_1, x_2 \\
c_6 &\geq &   \lambda_0 - 2 \lambda_1 + \lambda_2 + 4 \lambda_3  - 2 \lambda_4 & \leftarrow\ \mbox{Matching\ constant term} \\
\lambda_0,\ldots,\lambda_4 & \geq & 0 \\
\end{array}\]
Any nonzero solution yields a set of values for $\vc$ and the
corresponding Handelman representation for degree $2$. 
\end{example}

\paragraph{Putinar Representations:} Papachristadoulou and Prajna present
the Putinar representations approach to synthesizing Lyapunov functions.
Once again, we consider a semi-algebraic set $U$, as before. We fix
a form $V(\vc,\vx)$ for the Lyapunov and write 
\[ V  \equiv q_0 + \sum\limits_{j=1}^m q_j p_j \,.\]
for SOS polynomials $q_0,\ldots,q_m$. The approach fixes the degree of
each $q_j$ and uses SOS programming to encode the positivity.  The
result is a system of constraints over the parameters $\vc$ for $V$
and the unknowns $\vec{\lambda}$ that characterize the SOS multipliers
$q_j$. The same approach encodes the negative semi-definiteness of $V'$
over $U$. The combined result is a semi-definite program that jointly
solves for the positive definiteness of $V$ and the negative
definiteness of $V'$.  A solution is recovered by solving the
feasibility problem for an SDP to yield the values for $\vc$ that
yield a Lyapunov certificate for stability.

\section{Linear Programming relaxations based on Bernstein polynomials}\label{Sec:lp-Bernstein-relaxations}
In this section, we recall the use of Bernstein polynomials for
establishing bounds on polynomials in intervals.  Given a
multi-variate polynomial $p$, proving that $p$ is positive
semi-definite in $K$ is equivalent to showing that the 
optimal value of the following optimization problem  is non-negative:
\begin{equation}\label{eq:POP}
\begin{array}{ll}
\text{minimize} & p(\vx)\\
\text{s.t} & \vx \in K .
\end{array}
\end{equation}
Whereas (\ref{eq:POP}) is hard to solve, we will construct a linear
programming (LP) relaxation, whose optimal value is guaranteed to be a
lower bound on $p^*$.  If the bound is tight enough, then we can prove
the positivity of polynomial $p$ on $K$. 

In general, the Handelman representation approach presented in
Section~\ref{Sec:positivity-representations} can be used to construct a linear
programming relaxation~\cite{Sriram2013}.  In this section, we will
use Bernstein polynomials for the special unit box ($K=[0,1]^n$).
Bernstein polynomials extend the Handelman approach, and will be shown
to be strictly more powerful when $K$ is the unit box. In our
application examples, $K$ is often a hyper-rectangle but not
necessarily the unit box. We use an affine transformation to transform
$p$ and $K$ back to the unit box, so that the Bernstein polynomial
approach can be used.

\subsection{Overview of Bernstein polynomials}
Bernstein polynomials were first proposed by Bernstein as a
constructive proof of Weierstrass approximation
theorem~\cite{Bernstein/1912/Demonstration}. Bernstein polynomials are useful
in many engineering design applications for approximating geometric
shapes~\cite{Farouki/2012/Bernstein}. They form a basis for
approximating polynomials over a compact interval, and have nice
properties that will be exploited to relax the optimization~\eqref{eq:POP} to a linear program.
Here, we should mention that a relaxation using Bernstein polynomials was provided in the context of reachability analysis
for polynomial dynamical systems~\cite{DangSalinas09} and improved in~\cite{Bensassi/reach/2012}.
The novelty in this work is not only the adaptation of these relaxations in the context of polynomial Lyapunov function synthesis 
but also a new tighter relaxation will be introduced by exploiting the induction relation between Bernstein polynomials.
More details on Bernstein polynomials are available elsewhere~\cite{Munoz+Narkavicz/2013/Formalization}.
 
We first examine Bernstein polynomials and their properties for the univariate case and
then extend them to multivariate polynomials (see~\cite{Bernstein1,Bernstein2}).
\begin{definition}[Univariate Bernstein Polynomials]
Given an index $i \in \{0,\ldots,m\}$, the $i^{th}$ univariate Bernstein
polynomial of degree $m$ over $[0,1]$ is given by the following
expression:
\begin{equation}
 \beta_{i,m}(x)=\left(
\begin{array}{c}
 m \\ i
\end{array}
\right) x^i(1-x)^{m-i}, \quad i\in \{0,\dots,m\}.
\end{equation}
\end{definition}
In the Bernstein polynomial basis, a univariate polynomial $p(x): \sum\limits_{j=0}^m p_jx^j$ of
degree $m$, can be written as:
$$
p(x)=\displaystyle{\sum_{i=0}^{m}  b_{i,m} \beta_{i,m}(x) }
$$         
where for all $i=0,\dots,m$:
\begin{equation}\label{eq:bernstein-coeff}
 b_{i,m}=\sum_{j=0}^i \frac{\left(
\begin{array}{c}
 i \\ j
\end{array}
\right) 
}
{\left(
\begin{array}{c}
 m \\ j
\end{array}
\right) } p_j.
\end{equation}
The coefficients $b_{i,m}$ are called the \emph{Bernstein
  coefficients} of the polynomial $p$.

Bernstein polynomials have many interesting properties. We summarize the most relevant ones for our applications, below:
\begin{lemma}\label{lem:bernprop}
For all $x\in[0,1]$, and for all $m \in \mathbb{N}$, the Bernstein polynomials $\{ \beta_{0,m},\ldots,\beta_{m,m}\}$ have
the following properties: 
\begin{enumerate}
\item Unit partition: $\displaystyle{\sum_{i=0}^{m} \beta_{i,m}(x) }=1.$
\item Bounds: $0 \le \beta_{i,m}(x) \le \beta_{i,m}(\frac{i}{m}), \text{ for
  all }i=0,\dots,m.$
\item Induction property: $\beta_{i,m-1}(x)=\frac{m-i}{m}\beta_{i,m}(x)+\frac{i+1}{m}\beta_{i+1,m}(x), \text{ for all }i=0,\dots,m-1.$
\end{enumerate}
\end{lemma}
Using the unit partition and positivity of Bernstein polynomials, the following bounds result holds:
\begin{corollary}\label{cor:bernbound}
On the interval $[0,1]$, a polynomial $p$ with Bernstein coefficients $b_{0,m}, \ldots, b_{m,m}$, 
 the following inequality holds~\cite{Garloff93}:
\begin{equation}
 \displaystyle{\min_{i=0,\dots,m}b_{i,m}\le p(x) \le \max_{i=0,\dots,m}b_{i,m}}.
\end{equation}
\end{corollary}

We generalize the previous notions to the case of multivariate
polynomials i.e $p(\vx)=p(x_1,\dots,x_n)$) where $\vx=(x_1,\dots,x_n)\in U=[0,1]^n$. 
For multi-indices,
$I=(i_1,\dots,i_n)\in \mathbb{N}^n$,
$J=(j_1,\dots,j_n)\in \mathbb{N}^n$, we fix the following notation:
\begin{itemize}
\item $I \le J \; \iff \; i_l \le j_l,$  for all  $l=1,\dots n.$ 
\item $\frac{I}{J}=\left(\frac{i_1}{j_1},\dots,\frac{i_n}{j_n}\right)$ and 
 $\left(
\begin{array}{c}
 I \\ J
\end{array}
\right)=\left(
\begin{array}{c}
 i_1 \\ j_1
\end{array}
\right)\dots \left(
\begin{array}{c}
 i_n \\ j_n
\end{array}
\right).$
\item $I_{r,k}=(i_1,\dots,i_{r-1},i_r+k,i_{r+1},\dots,i_n)$ where $r\in\{1,\dots,n\}$ and $k\in \Z$.
\end{itemize}
Let us fix our maximal degree $\delta=(\delta_1,\dots,\delta_n)\in \mathbb{N}^n$ for a multi-variate polynomial $p$ ($\delta_l$ is the maximal degree of $x_l$ for all $l=1,\dots,n$).
Then the multi-variate polynomial $p$ will have the following form:
$$
p(\vx) = \sum_{I\le \delta} p_I \vx^I \text{ where } p_I \in \mathbb{R},\; \forall I\le \delta.
$$
Multivariate Bernstein polynomials are given by products of the univariate polynomials:
\begin{equation}
 B_{I,\delta}(\vx) = \beta_{i_1,\delta_1}(x_1) \dots \beta_{i_n,\delta_n}(x_n) \text{ where } 
\beta_{i_j,\delta_j}(x_j)=
\left(
\begin{array}{c}
 \delta_j \\ i_j
\end{array}
\right) 
x_j^{i_j} (1-x_j)^{\delta_j-i_j}.
\end{equation}
Thanks to the previous notations, these polynomials can also be written as follows: 
\begin{equation}
 B_{I,\delta}(\vx) =\left(
\begin{array}{c}
 \delta \\ I
\end{array}
\right) 
\vx^{I} (1_n-\vx)^{\delta-I}.
\end{equation}
Now, we can have the general expression of a multi-variate polynomial in the Bernstein basis:
$$
p(\vx)=\displaystyle{\sum_{I\le\delta }  b_{I,\delta} B_{I,\delta}(\vx) },
$$         
where Bernstein coefficients $(b_{I,\delta})_{I\le\delta}$ are given as follows:
\begin{equation}\label{eq:bernstein-coeff-formula}
b_{I,\delta}=\sum_{J\le I} \frac{\left(
\begin{array}{c}
 i_1 \\ j_1
\end{array}
\right) \dots
\left(
\begin{array}{c}
 i_n \\ j_n
\end{array}
\right) 
}
{\left(
\begin{array}{c}
 \delta_1 \\ j_1
\end{array}
\right) \dots
\left(
\begin{array}{c}
 \delta_n \\ j_n
\end{array}
\right) } p_J=\sum_{J\le I}\frac{\left(
\begin{array}{c}
 I \\ J
\end{array}
\right)}{\left(
\begin{array}{c}
 \delta \\ J
\end{array}
\right)}p_J.
\end{equation}
Therefore, the generalization of Lemma~\ref{lem:bernprop} will lead to the following properties:
\begin{lemma}
\label{lem:bernprop1}
For all $\vx=(x_1,\dots,x_n)\in U$ we have the following properties:
\begin{enumerate}
\item Unit partition: $\displaystyle{\sum_{I\le \delta} B_{I,\delta}(\vx) }=1.$
\item Bounded polynomials: $0 \le B_{I,\delta}(\vx) \le B_{I,\delta}(\frac{I}{\delta}), \text{ for all } I\le \delta.$
\item Induction relation: $B_{I,\delta_{r,-1}}=\frac{\delta_r-i_r}{\delta_r}B_{I,\delta}+\frac{i_r+1}{\delta_r}B_{I_{r,1},\delta}, \text{ for all } r=1,\dots,n., \text{ and all }I\le \delta_{r,-1}.$
\end{enumerate}
\end{lemma}
Finally, in the case of general rectangle $K=[\underline{x_1},\overline{x_1}]\times \dots \times [\underline{x_n},\overline{x_n}]$
it suffices to make a change of variables $x_j=\underline{x_j}+z_j(\overline{x_j}-\underline{x_j})$ for all $j=1,\dots,n$ to obtain new variables $\vz=(z_1,\dots,z_n)\in U$.

\subsection{Bernstein relaxations}
We assume that $K$ is a bounded rectangle. Without loss of generality,
we can assume that $K= [0,1]^n$ since we can be reduced to the unit
box by a simple affine transformation.  Using the previous properties
we are going to construct three LP relaxations to problem
(\ref{eq:POP}).

\paragraph{Reformulation Linearization Technique (RLT)} We first recall a simple approach to relaxing 
polynomial optimization problems to linear
programs~\cite{sherali91,sherali97}. We then carry out these
relaxations for Bernstein polynomials, and show how the properties in
Lemma~\ref{lem:bernprop1} can be incorporated into the relaxation
schemes. Recall, once again, the optimization problem~\eqref{eq:POP}.
\[\begin{array}{ll}
\text{minimize} & p(\vx)\\
\text{s.t} & \vx \in K .
\end{array}\]
For simplicity, let us assume $K: [0,1]^n$ be the unit rectangle. $K$
is represented by the constraints $K:\ \bigwedge\limits_{j=1}^n ( x_j
\geq 0\ \land (1- x_j) \geq 0 $.  The standard RLT approach consists
of writing $p(\vx) = \sum_{I} p_I\vx^I$ as a linear form
$p(\vx):\ \sum_{I} p_I y_I$ for \emph{fresh variables} $y_I$ that are
place holders for the monomials $\vx^I$. Next, we write down as many
\emph{facts} about $\vx^I$ over $K$ as possible. The basic approach
now considers all possible power products of the form
$\pi_{J,\delta}:\ \vx^J (1 - \vx)^{\delta-J}$ for all $J \leq \delta$,
where $\delta$ is a given degree bound. Clearly if $\vx \in K$ then 
$\pi_{J,\delta}(\vx) \geq 0$. Expanding
$\pi_{J,\delta}$ in the monomial basis as  $\pi_{J,\delta}:\ \sum_{I\leq \delta} a_{I,J} \vx^I$, we write the
linear inequality constraint,
\[ \sum_{I \leq J} a_{I,J}\ y_I \geq 0 \]
The overall LP relaxation is obtained as
\begin{equation}\label{eq:rlt-relax-pop}
\begin{aligned}
\text{minimize}& \sum_{I} p_I y_I \\
\text{s.t.} & \sum_{I \leq J} a_{I,J}\ y_I \geq 0,\ \mbox{for each}\ J \leq \delta \\
\end{aligned}
\end{equation}
Additionally, it is possible to augment this LP by adding inequalities
of the form $\ell_I \leq y_I \leq u_I$ through the interval evaluation
of $\vx^I$ over the set $K$.

\begin{proposition}
For any polynomial $p$, the optimal value computed by the LP~\eqref{eq:rlt-relax-pop} is 
a lower bound on that of the polynomial program~\eqref{eq:POP}.
\end{proposition}
\begin{example}
We suppose to compute a lower bound for the following POP:
\begin{equation}\label{ex:rlt}
\begin{array}{ll}
\text{minimize}&  {x_1}^2+{x_2}^2 \\
\text{s.t.} & (x_1,x_2)\in[0,1]^2 \\
\end{array}
\end{equation}
Using the RLT technique for a degree $\delta=2$ we denote by $y_{i,j}$ the fresh variables 
replacing the non linear terms $x^{(i,j)}={x_1}^i{x_2}^j$ for all $(i,j)\in \mathbb{N}^2$ such that $i+j \le2$.
Using these notations the objective function of the relaxation will be $ y_{20}+y_{02}$. The constraints
are given by the linearized form of the possible products (degree less than $\delta$) of the following constraints : $x_1\ge 0$, $x_2\ge 0$, $1-x_1\ge 0$
and $1-x_2\le 0$.
For example the constraint $y_{11}\ge 0$ is obtained by multiplying $x_1\ge 0$ and $x_2\ge 0$.
\end{example}

\paragraph{RLT using Bernstein Polynomials}
The success of the RLT approach depends heavily on writing ``facts''
involving the variables $y_I$ that substitute for $\vx^I$. We now
present the core idea of using Bernstein polynomial expansions and the
richer bounds that are known for these polynomials from
Lemma~\ref{lem:bernprop1} to improve upon the basic RLT approach. 

First, we write $p(\vx)$ as a sum of Bernstein polynomials of degree $\delta$.
\[ p(\vx): \sum_{I \leq \delta} b_{I,\delta} B_{I,\delta} \,, \]
wherein $b_{I,\delta}$ are calculated using the formula in
equation~\eqref{eq:bernstein-coeff-formula}.  Let us introduce a fresh
variable $z_{I,\delta}$ as a place holder for $B_{I,\delta}(\vx)$.
Lemma~\ref{lem:bernprop1} now gives us a set of linear inequalities
that hold between these variables $z_{I,\delta}$. Therefore, we obtain
three LP relaxations of increasing precision using Bernstein
polynomials. Once again, let $(b_{I,\delta})_{I\le \delta}$ denote
Bernstein coefficients of $p$ with respect to a maximal degree
$\delta$. We formulate three LP relaxations, each providing a better
approximation for the feasible region of the original
problem~\eqref{eq:POP}.

The first relaxation uses the fact that $B_{I,\delta}(\vx) \geq 0$ for all $\vx \in [0,1]^n$ and that
$\sum_{I \leq \delta} B_{I,\delta} \equiv 1$.
\begin{equation}
\label{eq:lpbern1}
\begin{array}{rllr}
{p_\delta}^{(1)}=&\text{minimize} & \displaystyle{\sum_{I\le \delta} {b}_{I,\delta} z_{I,\delta} }\\
&\text{s.t} & z_{I,\delta}\in \mathbb{R},\; & I\le \delta, \\
&&   z_{I,\delta} \ge 0,  & I\le \delta, \\
&& \displaystyle{\sum_{I\le \delta} z_{I,\delta} =1},
\end{array}
\end{equation}
From corollary~\ref{cor:bernbound}, it is easy to see that
$p_{\delta}^{(1)}=\displaystyle{\min_{I\le \delta} b_{I,\delta}}$ (minimal of Bernstein coefficients for
$p$). The optimization is superfluous here, but will be useful subsequently.

Next, we incorporate sharper bounds for $B_{I,\delta}(\vx)$ for each
$I$ for $\vx \in K$.
 \begin{equation}
\label{eq:lpbern2}
\begin{array}{rllr}
{p_\delta}^{(2)}=&\text{minimize} & \displaystyle{\sum_{I\le \delta} b_{I,\delta} z_{I,\delta} }\\
&\text{s.t} & z_{I,\delta}\in \mathbb{R},\; & I\le \delta, \\
&&  0 \le z_{I,\delta}\le B_{I,\delta}(\frac{I}{\delta}),  & I\le \delta, \\
&& \displaystyle{\sum_{I\le \delta} z_{I,\delta} =1},
\end{array}
\end{equation}

Finally, the recurrence relation between the polynomials are expressed
as equations in the LP relaxation to constrain the relaxation even
further.
%

\begin{equation}
\label{eq:lpbern3}
\begin{array}{rllr}
{p_\delta}^{(3)}=&\text{minimize} & \displaystyle{\sum_{I\le \delta} b_{I,\delta} z_{I,\delta} }\\
&\text{s.t} & z_{I,\delta}\in \mathbb{R},\;  I\le \delta, \\
&& z_{J,\delta'}\in \mathbb{R},\;   J\le \delta', \; \delta'< \delta,\\
&&  0 \le z_{I,\delta} \le B_{I,\delta}(\frac{I}{\delta}), \; I\le \delta, \\
&& 0 \le z_{J,\delta'} \le B_{J,\delta'}(\frac{J}{\delta'}), \; J\le \delta', \; \delta'< \delta,\\
&& \displaystyle{\sum_{I\le \delta} z_{I,\delta}  =1},\\
&& \displaystyle{\sum_{J\le \delta'} z_{J,\delta'}  =1},\; \delta' < \delta,\\
&& z_{J,\delta'}=\frac{{\delta'}_r-j_r}{{\delta'}_r}z_{J,{\delta'}_{r,1}}+\frac{j_r+1}{{\delta'}_{r}}z_{J_{r,1},{\delta'}_{r,1}},\; J\le \delta', \; \delta' < \delta.
\end{array}
\end{equation}

Relaxation (\ref{eq:lpbern1}) is obtained once the unit partition and
the positivity of Bernstein polynomials (Lemma~\ref{lem:bernprop1})
are injected.  In relaxation (\ref{eq:lpbern2}) we add lower bounds on
polynomials $B_{I,\delta}(y)$ (given by the second property of
Lemma~\ref{lem:bernprop1}) which allow us to obtain a more precise
result (greater) since we are adding more constraints for the previous
minimization problem. The third one (\ref{eq:lpbern3}) is obtained by
adding new variables for Bernstein polynomials of lower degree and
exploiting the induction property of Lemma~\ref{lem:bernprop1}. It
will be the more precise one but also the more costly.

We will show using the properties of Bernstein polynomials that
each of these relaxations provides a lower bound on the original
polynomial optimization problem. 
\begin{proposition}\label{prop:lpbern-relation}
 ${p_{\delta}}^{(1)} \le {p_{\delta}}^{(2)} \le {p_{\delta}}^{(3)} \le p^*$ where $p^*$ is the optimal value of (\ref{eq:POP}).
\end{proposition}
\begin{proof}


First, consider any feasible solution $y$ to the problem~\eqref{eq:POP} 
 $$
 \begin{array}{llr}
 \text{minimize} & \displaystyle{\sum_{I\le \delta} b_{I,\delta} B_{I,\delta}(y)} \\
 \text{s.t} & y\in [0,1]^n. \\
 \end{array}
 $$

 We note that replacing $z_{I} = B_{I,\delta}(y)$ the vector of all
 $z_I$s form a feasible solution to each of the three relaxations~\cref{eq:lpbern1,eq:lpbern2,eq:lpbern3}. Therefore,
$p_{\delta}^{(j)} \le p^*$ for $j \in \{1,2,3\}$.

Next considering the formulations
~\cref{eq:lpbern3,eq:lpbern2,eq:lpbern1}, we note that their
objective functions are the same.  Furthermore, the decision variables
in ~\cref{eq:lpbern2,eq:lpbern1} are the same; while the decision variables
in~\cref{eq:lpbern2} are a strict subset of those in~\cref{eq:lpbern3}.
Next, each constraint in~\cref{eq:lpbern1} is present in~\cref{eq:lpbern2},
and likewise, each constraint in~\cref{eq:lpbern2} is also present in~\cref{eq:lpbern3}.

As a consequence, any
feasible solution for $p_{\delta}^{(3)}$ is, in turn, a feasible
solution for $p_{\delta}^{(2)}$ with the extra variables in
formulation~\eqref{eq:lpbern3} removed.  Therefore, since the
objectives coincide and we seek to minimize, we have $p_{\delta}^{(2)}
\leq p_{\delta}^{(3)}$. Similarly, any feasible solution for
Eq.~\eqref{eq:lpbern2} is, in turn, a feasible solution for 
Eq.~\eqref{eq:lpbern1}. Here, no projection is needed, since
the two LPs consider the same set of variables. Once again,
we have $p_{\delta}^{(1)} \leq p_{\delta}^{(2)}$. Putting it all
together, we have 
${p_{\delta}}^{(1)} \le {p_{\delta}}^{(2)} \le {p_{\delta}}^{(3)} \le p^*$.

\end{proof}

\begin{remark}
\label{rmq;firstlevel}
The choice of the appropriate relaxation is a tradeoff between
complexity and precision. In fact, the third relaxation
(\ref{eq:lpbern3}) which gives the more precise result, can be very
expensive, especially when the number of variables and/or their degrees
increase. 

Finally,~\cref{eq:lpbern3} can be used for a \emph{fixed level} to
alleviate the drastic increase in the number of decision
variables. I.e, instead of exploiting all the constraints arising for
the degrees $\delta' < \delta$ we may restrict ourselves to $\delta'$
such that $\delta'_{r}=\delta_{r}-1$.
\end{remark}

\section{Comparison between Representations}\label{Sec:comparisons-relaxations}
In this section, we will first start by comparing linear and SOS representations.
Next, we compare the new Bernstein relaxations with other existing linear relaxations
including Handelman and interval representations.

\subsection{Comparison between Linear and SOS representations}
Comparing the presentations of linear vs. SOS representations,
the tradeoffs look quite obvious. Whereas linear representations
produce linear programs, that can be solved using exact arithmetic,
SOS representations produce sum-of-squares programs that are
solved numerically by relaxation to semi-definite programs.  In fact,
numerical issues sometimes arise, and have been noted in our previous
work~\cite{Sriram2013}. On the other hand,
it also seems that Handelman representations may be weaker than
Putinar representations. Consider the example below:

\begin{proposition}\label{prop:square-polynomial-fail}
The polynomial $p(x): x^2$ does not have a Handelman representation inside the 
interval $[-1,1]$.
\end{proposition}
\begin{proof}
 Suppose we were able to express
$x^2$ as a (non-trivial) conic combination of power products of the form
\[ x^2 \equiv \sum\limits_{j=1}^m \lambda_j (1 - x)^{n_j} (x+1)^{m_j}\,,\ \lambda_j > 0\]
We note that at $x = 0$, the LHS is zero whereas the RHS is strictly
positive.  This implies that $\lambda_j = 0$ for $j=1,\ldots,m$. Thus,
the RHS is identically zero, yielding a contradiction.
\end{proof}

On the other hand, the polynomial $x^2$ is SOS, and thus trivially
shown to be positive over $[-1,1]$ (if not over $\reals$) by a Putinar
representation.


%
%
%
%
%
However, in such a situation, we can show that Handelman representations can 
be useful in showing positivity where Putinar representations can fail.
Consider the set $K: [0,1]\times[0,1]$ and the bivariate polynomial
$p(x,y) = xy$. 
\begin{proposition}\label{prop:putinar-xy-fail}
There do not exist SOS polynomials $q_0, q_1, q_2, q_3, q_4$ such that 
\[ xy \equiv q_0 + q_1 x + q_2 y + q_3 (1-x) + q_4 (1-y) \,.\]
In other words, the polynomial $xy$ does not have a Putinar
representation over the unit box $K: [0,1]\times [0,1]$.
\end{proposition}
\begin{proof}
Suppose, for contradiction, there exist SOS polynomials $q_0, q_1, q_2, q_3, q_4$ such that 
\begin{equation}\label{eq:sos-forbidden}
xy \equiv q_0 + q_1 x + q_2 y + q_3 (1-x) + q_4 (1-y) \,.
\end{equation}

We establish a contradiction by considering the lowest degree terms of
the polynomials $q_0, \ldots,q_4$. We use the notation
$\coeffOf(q,x^iy^j)$ to refer to the coefficient of the monomial
$x^iy^j$ in the polynomial $q(x,y)$.

First, plugging in $x=0, y=0$, we observe that $\coeffOf(q_0, 1)
= \coeffOf(q_3,1) = \coeffOf(q_4,1) = 0$. In other words, the
constant coefficients of $q_0, q_3, q_4$ are zero.

Since $q_0, q_3, q_4$ are psd, if they have zero constant terms then 
they do not have linear terms involving $x$ or $y$.
\[ \coeffOf(q_j,x) = \coeffOf(q_j,y)= \coeffOf(q_j,1) = 0,\ j \in \{ 0,3,4\}\,.\]
Therefore, the constant terms of $q_1, q_2$ are zero as well:
\[ \coeffOf(q_1,1) = \coeffOf(q_2,1) = 0 \,.\]
Otherwise, the RHS will have non-zero terms involving $x,y$ but the LHS has no such terms.
Once again, from the positivity of $q_1,q_2$, we have
\[ \coeffOf(q_j,x) = \coeffOf(q_j,y) = 0,\ j \in \{ 1,2 \} \]
Having established that no constant or linear terms can exist for $q_0, \ldots, q_4$,
we turn our attention to the quadratic terms $x^2, y^2, xy$.  Consider the
coefficients of $x^2$ on both sides of Eq.~\eqref{eq:sos-forbidden}:
\[ \coeffOf(q_0,x^2) + \coeffOf(q_3,x^2) + \coeffOf(q_4,x^2) = 0,\   \coeffOf(q_0,y^2) + \coeffOf(q_3,y^2) + \coeffOf(q_4,y^2) = 0 \,. \]
Since $q_i$ are psd and lack constant/linear terms, we can show that
\[ \coeffOf(q_j,x^2) \geq 0,\ \coeffOf(q_j,y^2) \geq 0,\ j \in \{0,\ldots,4\} \]
Therefore, we conclude that
\[ \coeffOf(q_j,x^2) = \coeffOf(q_j,y^2) = 0,\ j \in \{ 0,3,4\} \,.\]
Finally, we compare $xy$ terms on both sides of Eq.~\eqref{eq:sos-forbidden} to obtain:
\[ \coeffOf(q_0, xy) + \coeffOf(q_3,xy) + \coeffOf(q_4,xy) = 1 \,.\]
Therefore, we have $\coeffOf(q_j,xy) >  0$ for some $j \in \{ 0,3,4\}$,
while $\coeffOf(q_j,x^2) = \coeffOf(q_j,y^2) = 0$. 
We now contradict the assumption that $q_j$ is psd. Based on what we have shown thus far, we can write
\[ q_j(x,y) = c xy + \mbox{third or higher order terms},\ \mbox{where}\ c >  0\,. \]
Let us fix $x=\epsilon, y = - \epsilon $ for some $\epsilon > 0$.
\[ q_j(\epsilon,-\epsilon )= - c_0 \epsilon^2 + o(\epsilon^3) \]
Therefore, we conclude for $\epsilon$ small enough,
$q_j(\epsilon,-\epsilon) < 0$, thus contradicting the positivity
assumption for $q_j$ for some $j \in \{ 0,3,4\}$. 

\end{proof}

%

Furthermore, the SOS relaxation to SDP relies on numerical interior point solvers
to find a feasible point. From the point of view of a \emph{guaranteed
method}, such an approach can be problematic. On the other hand, LP
solvers can use exact arithmetic in spite of the high cost of doing
so, to obtain results that hold up to verification. Examples of
numerical issues in SOS programming for Lyapunov function synthesis
are noted in our previous work~\cite{Sriram2013}, and will not be
reproduced here. Consequently, much work has focused on the problem of
finding rational feasible points for sum-of-squares to generate
polynomial positivity proofs that in exact
arithmetic~\cite{Harrison/2007/Verifying,Platzer+Quesel+Rummer/2009/Real,Monniaux+Corbineau/2011/Generation}.
Recently, a self-validated SDP solver VSDP has been proposed by Lange
et al.~\cite{Jannson/VSDP}.  However, its application to SOS
optimization has not been investigated.

\subsection{Comparison of Bernstein relaxations with other Linear Representations}
We are going to compare the Bernstein
relaxations~\cref{eq:lpbern3,eq:lpbern2,eq:lpbern1} with other
existing linear relaxations including Handelman and interval LP
relaxations. More precisely, we will show the benefit of using
Bernstein relaxations instead of the LP relaxations given by Ratschan et
al.~\cite{Ratschan+She/2010/Providing} and our previous
work~\cite{Sriram2013}.

Our earlier work~\cite{Sriram2013}, uses RLT with a combination of
Handelman representation augmented by interval arithmetic constraints
to prove polynomial positivity, as a primitive for Lyapunov function
synthesis. As long as the domain $K$ of interest is a hyper-rectangle,
the relaxations provided by Bernstein polynomials will provide results
that are guaranteed to be at least as good, if not strictly better.
For simplicity, let us fix $K$ as the unit box $[0,1]^n$ and compare
the two relaxations.

A first remark will be that the Handelman representation contains
polynomials with degree less or equal to the fixed degree $\delta$,
whereas our Bernstein polynomials are all of degree equal to
$\delta$. This does not affect the optimal value of the relaxation, as
noted by Sherali and Tuncbilek~\cite{sherali91}. Therefore, no gain of
precision can be made using the Handelman relaxation thanks to the
additional polynomials of degree less than $\delta$.  

\begin{lemma}[Bernstein vs. Handelman Representations]
Let $K: [0,1]^n$ represent the unit interval. Any polynomial
that can be shown nonnegative over $K$ using a Handelman representation of degree $\delta$ 
can also be shown nonnegative using the formulation in ~\cref{eq:lpbern1} with the same degree.
\end{lemma}
\begin{proof}
We first note that Handelman representation for $p$ seeks to express
$p$ as 
\[ p \equiv \sum_{I \leq \delta} \lambda_I \underset{B_{I,\delta}}{\underbrace{\vx^{I} (1-\vx)^{\delta-I}}} \,.\]
In fact, over the unit interval, the Handelman representation seeks to write $p$ as a conic
combination of Bernstein polynomials. 
Consider the relaxation to the POP:
\[ \min_{\vx \in K}\ p(\vx) \]
Using a Handelman representation of $p(\vx)$, we obtain the following relaxation:
\[\begin{array}{rllr}
{p_H}=&\text{minimize} & \displaystyle{\sum_{I\le \delta} {b}_{I,\delta} z_{I,\delta} }\\
&\text{s.t} & z_{I,\delta}\in \mathbb{R},\; & I\le \delta, \\
&&   z_{I,\delta} \ge 0,  & I\le \delta, \\
\end{array} \]

In contrast, we compare this  to the formulation~\cref{eq:lpbern1}, recalled below:
\[\begin{array}{rllr}
{p_\delta}^{(1)}=&\text{minimize} & \displaystyle{\sum_{I\le \delta} {b}_{I,\delta} z_{I,\delta} }\\
&\text{s.t} & z_{I,\delta}\in \mathbb{R},\; & I\le \delta, \\
&&   z_{I,\delta} \ge 0,  & I\le \delta, \\
&& \displaystyle{\sum_{I\le \delta} z_{I,\delta} =1},
\end{array} \]
Comparing the two LPs, it is easy to see that $p_H \leq p_{\delta}^{(1)} $. Also, 
if $p_H \geq 0$ then $p_{\delta}^{(1)} \geq 0$. Therefore, the result follows.
\end{proof}

\paragraph{Comparison with Interval Representations:} 
Now, we compare Bernstein relaxation to the interval relaxation over
$[0,1]^n$.  Interval relaxations are presented by Ratschan et
al.~\cite{Ratschan+She/2010/Providing} and in our previous
work~\cite{Sriram2013}. Notably, let $K$ be a hyper-rectangular
domain. The interval relaxation replaces each monomial $\vx^{I}$ with
an interval over $K$.  The interval for a polynomial $p$ is obtained
by summing up the interval for each term. While there exist many
approaches to evaluate a polynomial over an interval, we will consider
the scheme (implicitly) adopted by Ratschan et
al.~\cite{Ratschan+She/2010/Providing} and in our previous
work~\cite{Sriram2013} that uses an interval arithmetic based LP
relaxation for (parametric) polynomial optimization problems. Here we
will fix $K: [0,1]^n$, mapping arbitrary, bounded hyper-rectangles to
this domain through a linear change of variables (see
Section~\ref{Sec:transform-coordinates}).

\begin{lemma}
A polynomial $p: \sum_{I \leq \delta} c_I \vx^I$ can be shown to be non-negative over $[0,1]^n$ using
interval arithmetic if all its coefficients $c_I \geq 0$.
\end{lemma}

Following this, we note that if a polynomial $p: \sum_{I \leq \delta}
c_I \vx^I$ has $c_I \geq 0$, then all its Bernstein coefficients are
non-negative following~\cref{eq:bernstein-coeff}.

\begin{lemma}
  Any polynomial $p_I$ that can be shown non-negative over $K:
  [0,1]^n$ using interval arithmetic can also be shown non-negative
  using the Bernstein polynomial based formulation~\cref{eq:lpbern1}.
\end{lemma}
\begin{proof}
It can be  shown that the optimal value of~\cref{eq:lpbern1} is in fact the minimal
Bernstein polynomial coefficient, which has to be non-negative for $p_I$.
\end{proof}

\paragraph{Comparing Bernstein Relaxations:} Note that Prop.~\ref{prop:lpbern-relation} has
already demonstrated that any polynomial that can be shown nonnegative
over the unit interval by~\cref{eq:lpbern1} can be shown nonnegative by~\cref{eq:lpbern2}.
Likewise,~\cref{eq:lpbern3} is at least as powerful as~\cref{eq:lpbern2} in this
respect. Therefore, the major advantage
of using Bernstein polynomials is that, in addition to positivity, the
three non-trivial properties of Lemma~\ref{lem:bernprop1} can be used
to add linear relationships between the decision variables in the reformulation
linearization technique.  We first demonstrate that the second
relaxation~\eqref{eq:lpbern2} is strictly more powerful than the
relaxation in~\eqref{eq:lpbern1}.
\begin{example}\label{ex1}
Consider the simple univariate polynomial below: 
\begin{equation}
\mbox{Show that}\ p(x):\ 4x^2-4x+1 \geq 0 \ \text{ on } [0,1].
\end{equation} 

For this example, we find that the relaxation (\ref{eq:lpbern2}) with
a degree $2$ computes exact optimal value ${p_\delta}^{(2)}=p^*=0$,
proving positivity of $p$ over $[-1,1]$. However, ~\eqref{eq:lpbern1} yields
a minimal value of $-1$ and fails to prove positivity on $[0,1]$.
\end{example}

Now, we demonstrate that the third relaxation~\eqref{eq:lpbern3} is
strictly more powerful through an example.
\begin{example}\label{ex2}
We will consider the following bivariate polynomial:
\begin{equation}
p(x)=x^2+y^2 \text{ on } [-1,1]^2.
\end{equation}
For a degree $\delta=(2,2)$, the optimal value of (\ref{eq:lpbern1})
is ${p_{\delta}}^{(1)}=-2$, which does not establish positivity of $p$
on $[-1,1]$. If we  use the second linear program
(\ref{eq:lpbern2}), the optimal value will be improved and we find
${p_{\delta}}^{(2)}=-0.5$, but still not sufficient to prove positivity of
$p$ on $[-1,1]$. Now, when we use the third linear program
(\ref{eq:lpbern3}), we obtain the exact optimal value
${p_{\delta}}^{(3)}=0$ and ensure the positivity of $p$ over $[-1,1]^2$.
\end{example}

\section{Synthesis of polynomial Lyapunov functions}\label{Sec:poly-lyap-synth}

Given an ODE in the form: $\frac{d\vx}{dt} = f(\vx)$ with equilibrium
$\vx^*= 0$, we wish to find a Lyapunov function $V(\vx)$ over a given
rectangular domain $R_x$ containing $0$.

\begin{note}[Positive Semi-definite vs. Positive Definite]
As presented in Section~\ref{Sec:Lyapunov-synth}, our approach fixes a
 polynomial template $V_{\vc}(x)=V(\vx,\vc)$ for the target Lyapunov
 function, and computes its Lie derivative form $V'(\vx,\vc)$. It then
 searches for coefficients $\vc$ such that $V(\vx,\vc)$ is positive
 definite over $R_x$ and $V'$ is negative definite.  We recall the
 standard approach to encoding positive definiteness, following
 Papachristodoulou \&
 Prajna~\cite{Papachristodoulou+Prajna/2002/Construction}), by
 writing $V= U + \vx^{t} \Lambda \vx$ for a positive semi-definite
 function $U$ and a diagonal matrix $\Lambda$ with small but fixed
 positive diagonal entries.   Therefore, we will focus on encoding
 positive or negative semi-definiteness and use this approach to 
 extend to positive/negative definiteness.
\end{note}

We will now demonstrate how the three LP
relaxations~\cref{eq:lpbern1,eq:lpbern2,eq:lpbern3} described
in~\cref{Sec:lp-Bernstein-relaxations} extend to search for Lyapunov
functions, wherein
\begin{enumerate}[(a)] 
\item The polynomial of interest is $V(\vx,\vc)$ parameterized by unknowns $\vc$, 
\item The interval of interest is a general box $\prod\limits_{j=1}^n\ [\ell_j,u_j]$
rather than $[0,1]^n$,
\item We wish to encode the positive semi-definiteness of $-V'(\vx,\vc)$ rather than 
$V$ itself (following the technique in~\cref{Sec:simplified-encoding}).
\end{enumerate}

\subsection{Encoding Positivity of Parametric Polynomial}

We first consider the problem of extending the LP relaxation to find
values of parameters $\vc$, such that, a parametric polynomial
$\mathcal P(\vx,\vc)$ is positive semi-definite over the interval $[0,1]^n$.

Recall, that given a known polynomial $p(\vx)$,
our first step was to write down $p(\vx)$ using its Bernstein
expansion as $p(\vx):\ \sum_{I \leq \delta} b_{I} B_{I,\delta}$. Thus,
the overall form of~\cref{eq:lpbern1,eq:lpbern2,eq:lpbern3} can be
written as
\[ \min\ \sum_{I \leq \delta} b_I z_I\ \mbox{s.t.}\ A \vec{z} \leq \vec{b} \,.\]

However, Bernstein coefficients of a parametric polynomial $\mathcal P(\vx,\vc)$ are not known in advance.
Let $\vec{m}$ denote a vector of monomials $\vx^I$ for $I \leq \delta$.
The polynomial $\mathcal P(\vx,\vc)$ can be written as $\vc^t \cdot \vm$.
Furthermore, each monomial $\vx^I$ itself has a Bernstein 
expansion:
\[ \vx^{I}:\ \sum_{J \leq \delta} b_{J,I} B_{J,\delta} \,.\]
Consider a matrix $\scr{B}$, wherein, each row corresponds to a
monomial $\vx^I$, and each column to a Bernstein polynomial
$B_{J,\delta}$. The coefficient corresponding to row $I$ and column
$J$ is $b_{J,I}$, the Bernstein coefficient for $\vx^{I}$
corresponding to $B_{J,\delta}$. Therefore, we use $\scr{B}$ to
convert polynomials from monomial to the Bernstein basis.
\[ \mathcal P(\vx,\vc):\ \vc^{t} \vm = \vc^{t} \scr{B} \vz \,,\ \mbox{wherein}\ \vz\ \mbox{represents the Bernstein polynomials}\,.\]
Therefore, the LP relaxations~\cref{eq:lpbern1,eq:lpbern2,eq:lpbern3} have the following form:
\begin{equation}\label{eq:mpt-relaxation}
\begin{aligned}
\min\  & \vc^t\  \scr{B}\ \vz \\
\mbox{s.t.}\  & A \vz\ \leq  \vb \\
\end{aligned}
\end{equation}
Equation~\eqref{eq:mpt-relaxation} is, in fact, a bilinear program
which can be reformulated using its dual to a multiparamteric linear
optimization
problem\cite{Jones+Others/2007/Multiparametric,Kvasnica+Others/2004/Multi}. However,
the direct resolution of a multiparametric program is very expensive
since it requires to find exponentially many \emph{critical regions},
and for each region we have to find our optimal value which will be
an affine function depending on the parameter vector $\vc$.
Therefore, rather than solve the optimization problem ~\eqref{eq:mpt-relaxation}, we simply seek values of $\vc$ such that 
 \begin{equation}\label{eq:rlt-parameterized}
\mbox{find}\ \vc\ \mbox{s.t.}\ (\forall\ \vz)\ A \vz \leq \vb\ \Rightarrow\ \vc^t\ \scr{B}\ \vz \ \geq 0 \,. 
\end{equation}
We now use Farkas lemma, a well known result in linear programming, to dualize~\cref{eq:rlt-parameterized}.

\begin{lemma}\label{lem:feasibility}
 $\vc$ is a solution to the problem in~\cref{eq:rlt-parameterized} if and only if there exist multipliers $\vec{\lambda} \geq 0$
such that 
\[  A^t \vec{\lambda} = - \scr{B}^t\ \vc,\ \vb^t\ \vlam\ \leq 0,\ \mbox{and}\ \vlam \geq 0 \]
\end{lemma}

As a result, we now have a procedure to reduce the search for a
parametric positive polynomial as the feasibility problem for a set of linear
constraints. 
\begin{remark}
  The trick of using Farkas Lemma to handle multi-linear constraint is
   well known from previous work on the synthesis of ranking
  functions~\cite{Colon+Sipma/01/Synthesis,Podelski04}.
\end{remark}

\subsection{Simplified Encoding}\label{Sec:simplified-encoding}

Thus far, our approaches have encoded both the positive definiteness
of $V(\vx,\vc)$ and the negative definiteness of $V'(\vx,\vc)$ to
yield a combined linear or semi-definite program that can be used to
synthesize the Lyapunov function. In this section, we propose a
simplified approach that simply focuses on encoding the negative
definiteness of $V'(\vx,\vc)$, extracting a solution $\vc$ and
checking that the result $V_{\vc}(\vx)$ is in fact positive definite.
\begin{enumerate}
\item Choose a form $V(\vx,\vc)$.
\item Encode negative definiteness of $V'(\vx,\vc)$ (the Lie derivative) over $U$. In particular, we do not 
encode the positive definiteness of $V(\vx,\vc)$.
\item Compute a solution for $\vc$ and \emph{check} that the solution is, in fact, positive definite over $U$.
\end{enumerate}

The approach is motivated by the following result from Vannelli and
Vidyasagar (Page 72, Lemma 3)~\cite{Vannelli+Vidyasagar/1985/Maximal}.
A proof of this theorem is also included for the sake of completeness.
\begin{theorem}\label{theo:stab}
If $\scr{S}$ is an asymptotically stable system on $U$, $V(\vx)$
is a continuous function over $U$ with $V(\vec{0}) = 0$, and  $V'$ is
negative definite, then $V$ is positive definite in some neighborhood
of $\vec{0}$.
\end{theorem}
\begin{proof}
  Assume, for the sake of contradiction, that every neighborhood $N$
  of $\vec{0}$ has a point $\vx_0 \not=0$ such that $V(\vx_0) \leq
  0$. For each $N, \vx_0$, we will now show that the trajectory
  starting at $\vx_0$ cannot converge asymptotically to $\vec{0}$. Let
  $t \in [0,T)$ represent a time interval for which $\vx(t) \in
  N\setminus\{\vec{0}\}$. If $\vx(t) \in N\setminus\{0\}$ forever,
  then we set $T = \infty$. Consider any finite, or infinite sequence
  of time instances $t_0 = 0 < t_1 < t_2 \ldots < T$.  We observe that
  $0 \geq V(\vx(t_0)) > V(\vx(t_1)) > \cdots $, since
\[ V(\vx(t_i)) = V(\vx(t_{i-1})) + \underset{< 0}{\underbrace{\int_{t_{i-1}}^{t_i} V'(\vx(s)) ds}} \,.\]
By the continuity of $V$, and the fact that $V(\vec{0})= \vec{0}$, we
conclude that the trajectory $\vx(t)$ cannot converge asymptotically
to $\vec{0}$. In other words, the system is not asymptotically stable. This
directly contradicts our original claim.
\end{proof}

We will focus on encoding the
negative definiteness of $V'(\vx,\vc)$ over the given domain $R_x$,
without requiring that  $V(\vx,\vc)$ be positive definite. Once a suitable $\vc$ is
found, we simply check of $V_{\vc}(\vx)$ is positive definite over
$R_x$. Failing this, we simply choose a point $\vy \in R_x$ where $V$
fails to be positive and simply repeat our procedure by adding an additional constraint that
$V(\vy,\vc) > 0$.
\begin{remark}
The  advantage of this simplified encoding is that the synthesis part for $V_{\vc}$ is replaced 
by a simple check of positivity. If the obtained $V_{\vc}$ is non positive we can conclude using 
Theorem~\ref{theo:stab} that the system is already unstable. 

As a result, the simplified encoding results in a LP relaxation with fewer constraints.
\end{remark}

It now remains to address: (a) the transformation from 
a given domain $R_x$ to the domain $[0,1]^n$ for applying the Bernstein
polynomial based LP relaxations, and (b) encode negative definiteness of the
parametric polynomial $V'(\vx,\vc)$.

\subsection{Transforming Co-ordinates}\label{Sec:transform-coordinates}
Let $R_x: \prod\limits_{j=1}^n\ [\ell_j,u_j]$ be the domain of
interest.  We consider the change-of-basis  transformation
from $\vx \in R_x$ to a new set of variables $\vy \in [0,1]^n$
\[ x_j\ \mapsto\ \ell_j + y_j (u_j - \ell_j) \]
Let $\vm$ denote the original monomial basis over $\vx$ consisting
of monomials $\vx^I$ for $I \leq \delta$. Corresponding to this,
we define $\hat{\vm}$ as the monomial basis over $\vy$, consisting
of monomials $\vy^I$ for $I \leq \delta$. It is easy to see that
any monomial $\vx^I$ can be written as a polynomial involving monomials $\vy^J$
of degree $J$ at most $I$.  Therefore,
\[ {\vm} \equiv T \hat{\vm} \,\ \mbox{wherein} \]
each row of $T$ corresponds to a monomial $\vx^I$ and each column to a
monomial $\vy^J$.  Each row therefore lists the coefficients of the
monomial $\vx^I$ as a function over $\vy$.

Therefore, $\vc^t\ \vm\ \equiv\ \vc^t  T \hat{\vm} $. Rather than encoding
the positivity of the original polynomial $V(\vx,\vc):\ \vc^t \vm$ over $R_x$,
we encode that of $(T^t \vc)^t \hat{\vm}$ over $[0,1]^n$.

\subsection{Lie derivatives}

Finally, Lyapunov function synthesis requires us to encode the 
negative definiteness of $V'(\vx,\vc)$ rather than $V$.  Once again,
this requires us to consider the coefficients of the form $V'(\vx,\vc)$ 
as a linear transformation applied over $\vc$.

Since the RHS of the ODE is polynomial, we consider the Lie derivative
of each monomial $\vx^I$ as a polynomial $p_I$. Let $\scr{D}$
represent the matrix wherein each row of $\scr{D}$ represents the
monomial $\vx^I$ and the contents of the row are the coefficients of
the lie derivative of $\vx^I$.

Therefore, applying Lie derivative to $V(\vx,\vc):\ \vc^t \vm$, we obtain
\[ V'(\vx,\vc):\ \vc^t\ \scr{D}\ \vm' \,. \]
Here the vector $\vm'$ represents the set of monomials involved in the Lie derivative. 

\subsection{Overall Encoding}

To summarize, we are asked to find a value of $\vc$ such that the Lie
derivative of the polynomial $V(\vx,\vc)$ is non-negative over $R_x$.
Let $\scr{D}$ represent the matrix form of the Lie derivatives on the
monomial basis $\vm$, $T$ represent the transformation of the
monomials from $R_x$ to $[0,1]^n$, and finally $\scr{B}$ represent the
transformation to Bernstein form. The overall optimization involves
finding $\vc$ such that 
\begin{equation}
\mbox{find}\ \vc\ \mbox{s.t.}\ (\forall\ \vz)\ A \vz \leq \vb \ \Rightarrow\  (\scr{B}^t\ \times T^t\ \times \scr{D}^t \vc)^t \vz \leq 0
\end{equation}
As a result, applying Farkas lemma transforms this into solving the feasibility problem below:
\begin{equation}\label{Eq:lyapunov-derivative-encoding}
 \mbox{find}\ \vc\ \mbox{s.t.}\  (\exists\ \vlam)\ \underset{\mbox{LP feasibility}}{\underbrace{A^t \vlam = \scr{B}^t T^t \scr{D}^t \vc,\quad\ \vb^t\ \vlam \leq 0,\ \mbox{and}\ \vlam \geq 0}} \,.
\end{equation}
We note that $\vc = 0$ is seemingly a trivial solution to the
feasibility problem in ~\cref{Eq:lyapunov-derivative-encoding}. But,
this does not yield a Lyapunov function. To address, this, we recall
that our goal is to encode the \emph{negative definiteness} and not
the negative semi-definiteness of the derivative. On the other
hand,~\cref{Eq:lyapunov-derivative-encoding} encodes the negative
semi-definiteness. 

As mentioned earlier, we ensure that $U:\ V'(\vx,\vc)
- \vx' \Lambda \vx$ is negative semidefinite
using~\cref{Eq:lyapunov-derivative-encoding} rather than $V'$ itself.
The matrix $\Lambda$ is a diagonal matrix whose diagonal entries are
all set to a small value $\epsilon > 0$, chosen by the user. We choose
$\epsilon = 0.1$ for most of our experiments.

\begin{remark}
The problem posed in~\cref{Eq:lyapunov-derivative-encoding} can be
simplified considerably for the LP relaxation~\cref{eq:lpbern1}). In
the absence of further bounds about the Bernstein polynomials, the
smallest Bernstein coefficient is a lower bound on the minimum value
of a polynomial.  Therefore, the constraints
in~\cref{Eq:lyapunov-derivative-encoding} can be simplified as
\begin{equation}\label{eq:firstlp}
\mbox{find}\ \vc\ \mbox{s.t.  } \scr{B}\cdot \vc \geq 0\ \,.
\end{equation}
Effectively the form above imposes that all the Bernstein coefficients
of $V(\vx,\vc)$ are non-negative. This implicitly eliminates the
multipliers $\vlam$ from the LP relaxation.
\end{remark}

\begin{remark}
Infeasibility of ~\cref{Eq:lyapunov-derivative-encoding} means that
our search failed to find a Lyapunov function. This can indicate many
problems, including (a) the system is not stable, (b) the system is
stable but no polynomial Lyapunov function exists~\cite{Parillo2011}, (c) the system is stable
with a polynomial Lyapunov but it is not provable using the relaxation
that we have chosen to arrive at our LP.
\end{remark}

\subsection{Higher relaxation degree}
Our linear relaxations are based on a fixed degree for the Bernstein
polynomial expansion.  By default, this degree called $\delta$ is
fixed to some chosen value at the beginning of the algorithm. However,
if the technique fails to find a Lyapunov function, we may 
improve precision by increasing the degree $\delta$.  The following convergence
result motivates the possible improvement in the lower bounds of the LP relaxation
by increasing the degree bound $\delta$~\cite{Rokne1995}:
\begin{theorem}
Let $p$ be a multivariate polynomial of degree $\delta=(\delta_1,\dots,\delta_n)$ and let $b_{I,\delta}=b_I$ be its Bernstein coefficients with respect to the unit box $[0,1]^n$:
\begin{equation}
 \left| b_{I,\delta} - p\left(\frac{I}{\delta}\right)\right|=O \left(\frac{1}{\delta_1}+\dots+\frac{1}{\delta_n}\right) \text{ for all } I\le \delta.
\end{equation}
\end{theorem}
As a consequence, when the optimal value of our linear or bilinear
program is negative, we can just increase the degree of the relaxation
allowing the relaxation to be more precise and then increasing the
possibility to find our Lyapunov function.
\subsection{Branch and bound decomposition}
A second, more widely used approach to improving the relaxation, is to
perform a branch and bound decomposition. The essential idea consists
on verifying the so called \emph{vertex condition}~\cite{Garloff93} for the
given hyper-rectangle which guarantees that the LP relaxation
coincides with the optimal value. Informally, this condition requires that no local
minima for a polynomial $p$  exist in the interior of the rectangle. If it doesn't hold
we will simply divide our rectangle and keep doing it until reaching
the global minimum and getting exact bounds in each sub box. In our
case, we have two main differences:
\begin{enumerate}
\item We do not have a fixed polynomial, but a parametric polynomial $V(\vx,\vc)$.
\item The global minimum for a Lyapunov function $V$ is known in
  advance as the equilibrium $\vec{0}$. Likewise, the negation of its
  derivative also has $\vec{0}$ for a global minimum.
\end{enumerate}
For these reasons, our branch-and-bound approach focuses on
decomposing the given region $R_x$, so that the equilibrium $\vec{0}$
lies in the boundaries of our cells rather than the relative interior,
in an attempt to satisfy the \emph{vertex condition}.  So if a
Lyapunov function is not found, we simply choose a variable $x_j$ and
consider two cells $R_x^{(1)}:\ R_x \cap \{ x_j \leq 0 \}$ and
$R_x^{(2)}:\ R_x \cap \{ x_j \geq 0 \}$.  The cells may be recursively
subdivided if necessary. In the limit, this approach creates $2^n$
cells, and can be expensive for systems with more than $10$s of variables.
The computational complexity can be mitigated by examining a a few
cells in the decomposition and trying to find a Lyapunov candidate
based on the examined cells. We can then check if the Lyapunov
candidates are indeed Lyapunov functions by considering the other
cells. This approach can, in the worst case, examine every cell in the
decomposition. However, if a good empirical strategy for selecting the
cells can be found, the approach can save much effort involved in
encoding the LP relaxations for an exponential number of cells.

\section{Numerical results}\label{Sec:numerical-eval}

In this section, we present an evaluation of various linear
programming relaxations using Bernstein
polynomials~\cref{eq:lpbern1,eq:lpbern2,eq:lpbern3}, extended using
the technique for encoding the positivity of a parametric form,
presented in Section~\ref{Sec:poly-lyap-synth}.

\subsection{Implementation}

Our approaches are implemented as a MATLAB(tm) toolbox for
synthesizing Lyapunov function. Apart from a description of the system
to be analyzed, the inputs include the maximum degree $\delta:
(\delta_1,\ldots,\delta_n)$ for the Bernstein expansion in each
variable, the region of interest (fixed to $[-1,1]^n$ for all of our
evaluation), and the number of subdivisions along each dimension.
Furthermore, our toolbox implements three LP relaxations, each adding
more constraints over the previous. The first relaxation is based
on~\cref{eq:lpbern1} simply uses the non-negativity and the unit
summation properties of Bernstein polynomials. The second LP
relaxation is based on ~\cref{eq:lpbern2}, adds upper bounds to the
Bernstein polynomials and finally, the third
approach~\cref{eq:lpbern3} adds the recurrence relations between the
Bernstein polynomials. Each approach is used in the Lyapunov search by
encoding the dual form~\cref{Eq:lyapunov-derivative-encoding}.

\subsection{Numerical Examples}

We first compare and contrast the three LP relaxations here over some
benchmark examples from our previous work~\cite{Sriram2013}. Then using
using a special problem generator, we compare the results we obtain for each benchmark with those
obtained by using the \texttt{findlyap} function in
SOSTOOLS~\cite{sostools}, and the Lyapunov functions obtained in our
previous work.  For all the examples, we wish to prove asymptotic
stability over $R_x=[-1,1]^n$.  We will report for each program the
Lyapunov function, the number of boxes in our decomposition, and two
computational times.

\textsc{Setup} is the needed time to compute
the data for the linear program. This includes:
\begin{enumerate}
\item The time needed to compute the 
matrix $\scr{B}$  (for all three relaxations),
\item Computing bounds  on the Bernstein polynomials (for second and third relaxations), and
\item Time needed to compute recurrences for each Bernstein polynomial (for the third relaxation)
\end{enumerate}
In fact, much of the computation of $\scr{B}$ and the bounds on it are
independent of the actual problem instance.  They can be performed
once, and cached for a given number $n$ of variables and given degree
bounds $\delta$, instead of recomputing them separately for each
problem.

\textsc{LPTime} is the computational time associated with solving the
linear programming relaxation using the \texttt{linprog} function
provided by MATLAB(tm). Also, we should mention that since all the LPs
are feasibility problems, the objective function is set to be the
maximization of the sum of the coefficients.
\subsubsection{Benchmarks from~\cite{Sriram2013} and comparison with Handelman Representations}

\begin{table}[t]
\caption{Table showing Lyapunov functions computed by each of the three LP relaxations on the three systems considered in 
Example~\ref{Ex:1}. The column \textbf{Relaxation} indicates which of the three LP relaxations was used, the \textbf{Lyapunov}
function for each approach, the number of \textbf{Boxes} in the subdivision and the computational times split into computing
the matrices and linear programming data, and the actual time needed to solve the LP. All timings are in seconds. }\label{Tab:table-ex-1}
\begin{center}
\begin{tabular}{||l|l||l|l|l|l|}
\hline
System & Relaxation & Lyapunov  & \# Boxes & \textsc{Setup}   & \textsc{LPTime} \\
\hline
\eqref{eq:sys1} & LP1      &$4.5807x^2+4.5807xy +2.2906y^2$ & 2         & 0.06                   & 0.36                        \\
 & LP2 &$5x^2 +4.9995xy+2.5002y^2$          & 2         & 0.09                      & 0.34                  \\
 &LP3     & $5x^2 +5xy+2.5y^2$                        & 2        & 0.15                & 0.37                    \\
 &    &  &   &  & \\
\eqref{eq:sys2} & LP1     &$4.3039x^2 +4.3039y^2 $ & 4         & 0.13                 & 0.38                    \\
 & LP2 &$4.9998x^2 +5y^2$          & 4           &0.18                & 0.41                   \\
 & LP3    & $5x^2+5y^2$                   & 4     &0.37                 &  0.77                   \\ 
 &    &  &   &  & \\
\eqref{eq:sys3} & LP1  &$4.6809x^2 +4.9547y^2 $ & 4         & 0.16             & 0.36                     \\
  & LP2 &$4.9998x^2 +5y^2$          & 4           &0.18           & 0.40           \\
 &LP3     & $5x^2+5y^2$                   & 4     &0.33                &  0.43            \\
\hline
\end{tabular}
\end{center}
\end{table}

\begin{example}\label{Ex:1}
Consider the system over $(x,y)$: 

\begin{equation}\label{eq:sys1}
 \frac{dx}{dt} = - x^3 +y,\ \frac{dy}{dt} = -x-y \,.
\end{equation}

The Handelman relaxation technique in our previous
 work~\cite{Sriram2013} finds the Lyapunov function $x^2+y^2$ taking less than $0.1$ seconds, whereas
 SOS discovers $1.2118 x^2 + 1.6099\times 10^{-5} xy + 1.212 y^2$, requiring $0.4$ seconds. The
 three relaxations each using a subdivision of $[-1,1]^2$ discover the
 function $x^2+ xy + \frac{1}{2} y^2$ (with a multiplicative factor,
 and modulo small perturbations due to floating point
 error). Interestingly, the system is globally asymptotically stable,
 and the Lyapunov function discovered by our approach is valid
 globally.

Next, we consider the system:
\begin{equation}\label{eq:sys2}
\frac{dx}{dt} = -x^3 - y^2,\ \frac{dy}{dt} = xy - y^3 \,.
\end{equation}

The Handelman relaxation approach~\cite{Sriram2013} finds a $4$ degree
Lyapunov function $x^4 + 2 x^2 y^2 + y^4$, requiring less than $0.1$
seconds, whereas the SOS approach produces $0.62788 x^4 + 0.052373 x^3
+ 0.65378 x^2 y^2 + 1.1368 x^2 - 0.18536 x y^2 + 0.60694 y^4 + 1.1368
y^2$ after deleting terms with coefficients less than $10^{-7}$.  The
SOS approach requires roughly $0.4$ seconds for this example. Our
approach discovers degree two Lyapunov function $x^2 + y^2$ that is
also globally stable.

Finally, we consider the system:

\begin{equation}\label{eq:sys3}
\frac{dx}{dt} = -x - 1.5 x^2 y^3,\ \frac{dy}{dt} =  -y^3 + 0.5 x^2 y^2 \,.
\end{equation}

The approach in~\cite{Sriram2013} proves asymptotic stability over
$[-1,1]^2$ through the function $0.2x^2+y^2$, requiring  $0.4$
seconds, whereas the SOS approach finds $2.4229 x^2 + 4.4868 y^2$
requiring a running time of $8.8$ seconds.

The specific Lyapunov functions found for
systems~\cref{eq:sys1,eq:sys2,eq:sys3}, the running times and number
of subdivisions needed are summarized in
Table~\ref{Tab:table-ex-1}.  

\end{example}

\begin{table}[ht]
\caption{Performance of our approach on the synthesized benchmarks. The column $n$: number of variables,
$d_{\max}$: maximum degree of the vector field, \textsc{succ}?
indicates whether the approach succeeded in finding a Lyapunov
function, $\tick$: succeeded with Lyapunov, \textsc{np}: numerical
problem, \textsc{mo}: out-of-memory, $d_L$: degree of Lyapunov
function, $d_Q$: degree of SOS multipliers, \textsc{Setup}: setup
time, $T_{SDP}$: SDP Solver time, \textsf{Rel. Typ.}: Relaxation
Type, \#Box: number of boxes in decomposition, $T_{LP}$: LP solver
time. All times are reported in
seconds.}\label{Tab:synth-benchmarks-data}
\begin{center}
{\small
\begin{tabular}{|l | l  l | l  l l l  l | l  l  l  l l |}
\hline
ID & $n$ & $d_{\max}$ & \multicolumn{5}{c|}{Putinar  (SOS)} & \multicolumn{5}{c|}{Bernstein (our approach)} \\
\cline{4-13}
   &         &      &  \textsc{succ}? & $d_{L}$ & $d_{Q}$ & \textsc{Setup} &  $T_{SDP}$ & Rel. Typ. & \textsc{succ}? & \# Box & \textsc{Setup} & $T_{LP}$ \\
\hline

1 & 2 & 3  & \tick & 2 & 2 & 0.35 &  0.9 & LP1 & \tick &4   &0.17   & 0.43 \\
  &   &   &        &   &   &      &      & LP2 &\tick  &4   &0.20   & 0.42  \\
  &   &   &        &   &   &      &      & LP3 &\tick  &2   & 0.17  &0.38   \\[10pt]
2 & 2 & 3 &  \tick  & 2 & 2 & 0.3 & 0.67 &  LP1 &\tick &4  &0.17   &0.42     \\
  &   &   &        &   &    &    &      &  LP2 &\tick  & 4 &0.19  & 0.38  \\
  &   &   &        &   &    &   &      &  LP3 &\tick  &  2 & 0.17  &0.35   \\[10pt]
3 & 2 & 3 &  \tick  & 2 & 2 & 0.33 & 0.61 &  LP1 &\tick  & 4  &0.17   & 0.37  \\
  &   &   &        &    &   &      &      &  LP2 &\tick &4  & 0.18  & 0.35  \\
  &   &   &        &    &   &     &      &   LP3 &\tick&2   &0.16   &0.35   \\[10pt]
4 & 2 & 3 &  \tick  & 2 & 2 & 0.3 & 0.97 &  LP1 &\tick  &4   &0.17   &0.37   \\
  &   &   &        &    &   &     &      &  LP2 &\tick  &4   &0.21   &0.39   \\
  &   &   &        &    &    &     &      &  LP3 &\tick  &2   &0.17   &0.35   \\[10pt]
5 & 3 & 3 &  \tick  & 2  & 2 &  0.86 & 1.12 &  LP1&\tick &8 &0.81 &0.47   \\
  &   &   &        &       &     &   &      &  LP2 &\tick  &8   & 0.97  &0.61   \\
  &   &   &        &    &    &       &      &  LP3 &\tick  &4   &1.24   &0.70   \\[10pt]
6 & 3 & 5 & \cross (\textsc{np}) & 2 & 2 &  0.81 & 2.3 & LP1 &\tick &8   &7.15   &6.4   \\
  &   &   & \tick   &  2 &  4 &  39.5   &  4.2     &  LP2 &\tick &8& 7.83  &17.17    \\
  &   &   &         &    &    &         &          &  LP3 &\cross(\textsc{np}) &8 &17.4   &102.3   \\[10pt]
7 & 3 & 5 &  \cross  (\textsc{np}) & 2 & 2 & 0.8  & 2.2 & LP1 &\tick &8   &6.50   &5.2  \\
  &   &   &  \cross (\textsc{np}) & 2 & 4 & 40    & 4.6 &  LP2 &\tick  &8   &7.42   &5.7   \\
  &   &   &  \cross(\textsc{np}) & 4 & 4 &  40.5   &  7 &  LP3 &\tick &8  &13.2  & 26.8  \\

\hline
\end{tabular}
}
\end{center}
\end{table}

\subsubsection{Synthetic Benchmarks and comparison with SOS}
 
We now consider a second class of \emph{synthetic benchmarks} that
were generated using a special problem generator, constructed for
generating challenging examples of locally stable polynomial vector
fields of varying degrees and number of variables to evaluate the
various techniques presented here.  Our overall idea is to fix two
homogeneous polynomials $V_1(\vx)$ and $V_2(\vx)$ that are positive
definite over a region of interest, chosen to be $K: [-1,1]^n$ for our
examples. The benchmarks described in this section along with the
Lyapunov functions synthesized are available on-line through arXiv~\cite{BenSassi+Others/2014/Synthesis}.

Subsequently, for each choice of $V_1, V_2$, we attempt to find a
system $\frac{d\vx}{dt} = F(\vx)$ such that the Lie derivative of
$V_1$ is $-V_2$, and with an equilibrium at $\vec{0}$.
\begin{equation}\label{Eq:benchmark-synth-proc}
(\grad V_1) \cdot F =  - V_2,\ \mbox{and}\ F(\vec{0}) = \vec{0} \,.
\end{equation}
Naturally, any such system using the vector field $F$ is guaranteed to
be asymptotically stable due to the existence of $V_1, V_2$.  To
synthesize a benchmark that is guaranteed to have asymptotic
stability, we need to find a suitable $F$ within a given degree bound.
To this end, we parameterize our system $F$ by a set of parametric
polynomials and attempt to find parameters that
satisfy~\cref{Eq:benchmark-synth-proc}. It is easy to show that our
approach leads to a set of linear equations on the parameters defining
the entries in $F$ and solving these equations yields a suitable
system $F$. The difficulty here lies in choosing appropriate $V_1,V_2$
so that the system $F$ can be found. In our experience, if $V_1, V_2$
are chosen arbitrarily, the likelihood of finding a function $F$ that
satisfies~\cref{Eq:benchmark-synth-proc} seems quite
small. Furthermore, since $F$ involves $n$ polynomials, the technique
yields prohibitively large equations for $n \geq 6$.  Our approach to
synthesize benchmarks is based on carefully controlling the choice of
$V_1,V_2$ and repeated trial-and-error, until feasible system of
equations is discovered, to synthesize a benchmark.  Having
synthesized our benchmark, we \emph{hide} the functions $V_1,V_2$
used to generate it and simply present the system $F$ to our
implementation, as well as for SOS program.

The key  to finding benchmarks lies in the generation of the polynomials $V_1, V_2$.
We generated $V_1$ as one of two simple forms: (a) $V_1: \vx^t \Lambda_1 \vx$, or
(b)  $V_1:\ \vm^t \Lambda_2 \vm$, 
wherein $\vm$ is a vector of squares of the system variables of the
form $[x_1^2, \ x_2^2,\  \cdots,\  x_n^2]'$, and $\Lambda_1,\Lambda_2$ are
diagonal matrices with non-negative diagonal entries chosen at random.

The polynomial $V_2$ is chosen to be a positive definite polynomial
over $[-1,1]^n$. The key idea here is to generate $V_2$ that is guaranteed
to be positive definite over $[-1,1]^n$ by writing 
\[ V_2 (\vx):\ \vx^t \Lambda \vx + \sum_j q_j \prod_{i=1}^n (1+x_i)^{p_{j,i}} (1-x_i)^{q_{j,i}} \,, \]
essentially as a Schm{\"u}dgen representation involving the polynomials
$(1-x_i), (1+x_i)$ for $i \in [1,n]$ and \emph{sum-of-squares}
polynomials $q_j$ obtained by squaring and adding randomly generated
polynomials together.

\begin{remark}
Even though our approach synthesizes an ODE $\frac{d\vx}{dt} = F(\vx)$
that by design has a Lyapunov function $V(\vx)$, we note that the
resulting system may (and often does) admit many other Lyapunov
functions with a possibility of a larger domain of attraction towards
the equilibrium $\vec{0}$.
\end{remark}

In many cases, the process of trial and error is required to find
pairs $V_1, V_2$ that yield a feasible vector field. Using this
process, $15$ different benchmarks were synthesized with $5$ each of
degrees $2,3,$ and $4$, respectively. 
\iftrversion
Appendix~\ref{App:description}
reports the ODEs for these benchmarks and the Lyapunov functions
synthesized by our technique.
\fi

\begin{table}[h!]
\caption{Performance of our approach on the synthesized benchmarks (continued). Note that \textsc{mo}: out-of-memory termination, \textsc{to}: time-out. All times are reported in seconds.}\label{Tab:synth-benchmarks-data-ext}
\begin{center}
{\small
\begin{tabular}{|l | l  l | l  l l l  l | l  l  l  l l |}
\hline
ID & $n$ & $d_{\max}$ & \multicolumn{5}{c|}{Putinar  (SOS)} & \multicolumn{5}{c|}{Bernstein (our approach)} \\
\cline{4-13}
   &         &      &  \textsc{succ}? & $d_{L}$ & $d_{Q}$ & \textsc{Setup} &  $T_{SDP}$ & Rel. Typ. & \textsc{succ}? & \# Box & \textsc{Setup} & $T_{LP}$ \\
\hline
8 & 3 & 5 & \cross (\textsc{np})& 2 & 2 & 0.8 & 1.7 &LP1 &\tick & 8&10.63   & 10.9   \\
  &   &   & \cross (\textsc{np})&2 & 4 &  40.9  & 7.9     &  LP2 & \tick & 8 &11.91   & 30.97   \\
  &   &   & \cross (\textsc{np})&4 & 4 & 40.1     & 5.5     &  LP3 & \cross (\textsc{np})  &8   &22.38   &130.77   \\[10pt]
9 & 3 & 2 &  \cross (\textsc{np}) & 2 & 2 & 0.9 & 4.1 &LP1 &   \cross (\textsc{np})  & 8 &1.99   &0.61   \\
  &   &   & \cross (\textsc{np})& 2 & 4 & 42.2 & 3.7 &  LP2 &\tick &  8 &  2.06 & 0.92  \\
  &   &   & \tick   & 4  & 4 & 41.9 & 3.1 & LP3 & \tick  &8   &3.04   &3.81   \\[10pt]
10 & 3 &5 &  \cross (\textsc{np}) &2 &2 &	0.9 & 2.9 & LP1 &  \cross (\textsc{np}) &8  &3.48   &3.19   \\
   &   &   &  \cross (\textsc{np}) & 2 & 4 & 38.3 & 5.3 & LP2 &\tick &8   &1.23   &1.88   \\
  &    &    & \tick      & 4  & 4  & 38.7  & 5.54     & LP3 &\tick  & 8  &1.56   &0.60 \\[10pt]	
11 & 4 & 3 & \tick& 2 & 2 & 3.7 & 3.1 & LP1 & \tick &16 &3.58   &3.25     \\
  &    &    &       &   &   &      &     & LP2 &\tick &16  & 4.34  &17.27  \\							
  &    &    &       &   &   &      &     & LP3 & \tick & 16 & 9.05  & 53.5 \\[10pt]						
12 & 4& 3 & \cross (\textsc{np}) & 2 &2	 & 3.7	& 2.1 & LP1&\tick&16& 5.16 &16.85   \\
   &  &   & \cross (\textsc{mo})	& 2 & 4 & \multicolumn{2}{c|}{$> 600$} & LP2 &\tick&16  &6.38   &12.86    \\
   &  &    &       &   &   &      &     & LP3 &  \cross (\textsc{np}) &  16  &22.23   &224.23 \\[10pt]							
13 & 4 &6 & \cross (\textsc{np})  & 2 & 2 & 4& 3.1 & LP1 & \cross (\textsc{np})&16 &41.36   &627.25     \\
&  &   & \cross (\textsc{mo})	& 2 & 4 & \multicolumn{2}{c|}{$> 600$} & LP2 & \cross(\textsc{np}) &16   &43.38   &988.31   \\
 &    &    &       &   &   &      &     & LP3 & \cross (\textsc{to})  &16 &   \multicolumn{2}{c|}{$> 1200$}  \\[10pt]							
14 & 4 &6 & \cross (\textsc{np})  & 2 & 2 & 3.8 & 3.6 & LP1 &\cross (\textsc{np})  &16& 37.45  &339.86      \\
&  &   & \cross (\textsc{mo})	& 2 & 4 & \multicolumn{2}{c|}{$> 600$} & LP2 &\cross(\textsc{np})  &16 &41.93   & 1049.53     \\
 &    &    &       &   &   &      &     & LP3 & \cross (\textsc{to})  &  16 &   \multicolumn{2}{c|}{$> 1200$}   \\[10pt]	
15 & 4 &6 & \cross (\textsc{np})  & 2 & 2 & 3.8 & 3.9 & LP1 &\cross (\textsc{np}) &16 &38.55   &368.48     \\
&  &   & \cross (\textsc{mo})	& 2 & 4 & \multicolumn{2}{c|}{$> 600$} & LP2 &\cross(\textsc{np})  &16   &45.32   &888.33   \\
 &    &    &       &   &   &      &     & LP3 &   \cross (\textsc{to})&16 &  \multicolumn{2}{c|}{$> 1200$}   \\													
\hline

\end{tabular}
}
\end{center}
\end{table}

\subsubsection{Results} Tables~\ref{Tab:synth-benchmarks-data} and ~\ref{Tab:synth-benchmarks-data-ext}
compare the performance of the three LP relaxations implemented in our
prototype with an implementation \textsf{Putinar (SOS)}, based on
Putinar representation of the Lyapunov function and the negation of
its derivative, built using SOSTOOLS. Here we should mention that, in order to reduce the complexity of the `LP3' 
relaxation, we reduce ourselves to a first level of lower degrees (see Remark~\ref{rmq;firstlevel}).
For each of the $15$ benchmarks, we run both tools under different
setups. The \textsf{Putinar (SOS)} approach is run with varying
degrees of the Lyapunov function $d_L$, and degrees of the SOS
multipliers $d_Q$. We attempted three sets $(d_L, d_Q)= (2,2), (2,4),
(4,4) $ in succession, stopping as soon as a Lyapunov function is
found without a failure. To experiment with our approach and enable a
full comparison, we attempt all the three relaxations for all the
benchmarks.

We note that the LP relaxation approach is generally successful in
discovering Lyapunov functions.   In 7 out of 15 cases, all three
LP relaxations succeed, while at least one LP relaxation succeeds
in 12 out of 15 cases. On the other hand, the \textsf{Putinar (SOS)}
approach succeeds in 9 out of the 15 attempts, with
\emph{numerical problems} (\textsc{np}) being the most common
failure mode. These may arise due to many reasons, but commonly due to
the Hessian matrix becoming ill-conditioned during the calculation of
a Newton step.  For benchmarks 12-15, the polynomials involved become
so large, that the \textsf{Putinar (SOS)} approach runs out of memory
during the problem setup, causing MATLAB(tm) to crash. Our approach
also suffers from the same set of problems, but to a noticeably lesser
extent. For instance, 11 out of the 45 linear programs failed due to
numerical problems, and 3 more due to timeouts. On the other hand, 15
out of the 29 SDPs terminate with a numerical problem with an
additional 4 out-of-memory issues.

On most of the smaller benchmarks, all approaches have comparable
timings.  In general, the third relaxation (LP3) is the most
expensive, often more expensive than the other two LP relaxations or
the \textsf{Putinar (SOS)} approach. Likewise, when the degree of the
SOS multipliers $d_Q$ is increased from $2$ to $4$, we witness a
corresponding $40\times$ factor increase in the time taken to setup the
SDP, with a smaller increase in the time taken to solve the SDP.  For
the larger examples, the LP relaxation requires more time, but is
generally successful in finding an answer.

Finally, all approaches fail on benchmarks
13-15. 
\iftrversion
Appendix~\ref{App:description} shows these benchmarks. 
\fi
 A key
issue is the blowup in the number of monomial terms to be considered
in the parametric polynomial forms for the Lyapunov function and its
derivatives. This blowup seems to overwhelm both our approach and the
SOS programming approach. We conclude that handling large parametric
polynomials efficiently remains a challenging problem for our approach
as well as the \textsf{Putinar (SOS)} approach.

\section{Conclusion}
To conclude, we have examined three different LP relaxations for
synthesizing polynomial Lyapunov functions for polynomial systems. We
compare these approaches to the standard approaches using
Schm{\"u}dgen and Putinar representations that are used in SOS
programming relaxations of the problem. In theory, the Schm{\"u}dgen
representation approach subsumes the three LP relaxations. In
practice, however, we are forced to use the Putinar representation. We
show that the Putinar representation can prove some polynomials
positive semi-definite that our approaches fail to. On the other hand,
the reverse is also true: we demonstrate a polynomial that is easily
shown to be positive semi-definite on the interval $[-1,1]^n$ through
LP relaxations. However, the same fact cannot be demonstrated by a
Putinar representation approach. We then compare both approaches over
a set of numerical benchmarks. We find that the LP relaxations succeed
in finding Lyapunov functions for all cases, while the Putinar
representation fails in many benchmarks due to numerical (conditioning)
issues while solving the SDP. As future work, we wish to extend our
approach to a larger class of Lyapunov functions.  We also are looking
into the problem of analyzing systems with non-polynomial dynamics and
the synthesis of non-polynomial Lyapunov functions. 

Finally, Lyapunov functions have, thus far, remained important
theoretical tools for analyzing the stability of control
systems. However, these tools are seldom used, in practice, for 
industrial scale systems. This is chiefly due
to the burden of manually specifying Lyapunov functions. Therefore, stability of
complex industrial designs are often ``verified'' by extensive
simulations.  Recent work by Kapinski et al. argues that Lyapunov
functions can be of practical values for automotive designs, provided
they can be discovered easily, and certified using formal verification
tools~\cite{Kapinski+Others/2014/Simulation}.  We hope that the use of
linear relaxations can provide us with more approaches to synthesize
Lyapunov functions that can serve as certificates for stability.

\bibliographystyle{plainnat}

\begin{thebibliography}{64}
\providecommand{\natexlab}[1]{#1}
\providecommand{\url}[1]{\texttt{#1}}
\expandafter\ifx\csname urlstyle\endcsname\relax
  \providecommand{\doi}[1]{doi: #1}\else
  \providecommand{\doi}{doi: \begingroup \urlstyle{rm}\Url}\fi

\bibitem[Ahmadi and Majumdar(2014)]{Ahmadi+Majumdar/2014/DSOS}
A.A. Ahmadi and A.~Majumdar.
\newblock {DSOS} and {SDSOS} optimization: {LP} and {SOCP}-based alternatives
  to sum of squares optimization.
\newblock In \emph{Intl. Conference on Information Sciences and Systems
  (CISS)}, pages 1--5. {IEEE} Press, March 2014.

\bibitem[Ahmadi et~al.(2011)Ahmadi, Krstic, and Parrilo]{Parillo2011}
Amir~Ali Ahmadi, Miroslav Krstic, and Pablo~A. Parrilo.
\newblock A globally asymptotically stable polynomial vector field with no
  polynomial {L}yapunov function.
\newblock In \emph{CDC-ECE}, pages 7579--7580, 2011.

\bibitem[Bernstein(1952)]{Bernstein1}
S.~Bernstein.
\newblock \emph{Collected Works}, volume~1.
\newblock USSR Academy of Sciences, 1952.

\bibitem[Bernstein(1954)]{Bernstein2}
S.~Bernstein.
\newblock \emph{Collected Works}, volume~2.
\newblock USSR Academy of Sciences, 1954.

\bibitem[Bernstein(1912)]{Bernstein/1912/Demonstration}
Sergei~Nanatovich Bernstein.
\newblock D{\'e}monstration du th{\'e}or{\'e}me de {W}eierstrass fond{\'e}e sur
  le calcul des probabilit{\'e}s.
\newblock \emph{Communcations de la Soci{\'e}t{\'e} Math{\'e}matique de
  Kharkov\ 2}, \penalty0 (1):\penalty0 1--2, 1912.

\bibitem[Bernstein(1915)]{Bernstein1915}
Sergei~Natanovich Bernstein.
\newblock On the representation of positive polynomials.
\newblock \emph{Soobshch Kharkov marem ob-va}, 2\penalty0 (14):\penalty0
  227--228, 1915.

\bibitem[Chesi(2005)]{Chesi/2007}
G.~Chesi.
\newblock Estimating the domain of attraction via union of continuous families
  of {L}yapunov estimates.
\newblock \emph{Systems and Control letters}, 56\penalty0 (4):\penalty0
  326--333, 2005.

\bibitem[Chesi(2010)]{Chesi/2009}
G.~Chesi.
\newblock Polynomial relaxation-based conditions for global asymptotic
  stability of equilibrium points of genetic regulatory networks.
\newblock \emph{International Journal of Systems Science}, 41\penalty0
  (1):\penalty0 65--72, 2010.

\bibitem[Chesi et~al.(2005)Chesi, Garulli, Tesi, and Vicino]{Chesi+Vicino/2005}
G.~Chesi, A.~Garulli, A.~Tesi, and A.~Vicino.
\newblock {LMI}-based computation of optimal quadratic {L}yapunov functions for
  odd polynomial systems.
\newblock \emph{International Journal of Robust and Nonlinear Control},
  15\penalty0 (1):\penalty0 35--49, 2005.

\bibitem[Collins(1975)]{Collins/75/Quantifier}
G.E. Collins.
\newblock Quantifier elimination for real closed fields by cylindrical
  algebraic decomposition.
\newblock In H.Brakhage, editor, \emph{Automata Theory and Formal Languages},
  volume~33 of \emph{Lecture Notes in Computer Science}, pages 134--183.
  Springer, 1975.

\bibitem[Collins and Hong(1991)]{Collins+Hong/91/Partial}
George~E. Collins and Hoon Hong.
\newblock Partial cylindrical algebraic decomposition for quantifier
  elimination.
\newblock \emph{Journal of Symbolic Computation}, 12\penalty0 (3):\penalty0
  299--328, sep 1991.

\bibitem[Colon and Sipma(2001)]{Colon+Sipma/01/Synthesis}
Michael Colon and Henny Sipma.
\newblock Synthesis of linear ranking functions.
\newblock In Tiziana Margaria and Wang Yi, editors, \emph{Tools and Algorithms
  for Construction and Analysis of Systems}, volume 2031, pages 67--81.
  Springer, April 2001.

\bibitem[Dang and Salinas(2009)]{DangSalinas09}
T.~Dang and D.~Salinas.
\newblock Image computation for polynomial dynamical systems using the
  {Bernstein} expansion.
\newblock In \emph{CAV'09}, volume 5643 of \emph{LNCS}, pages 219--232.
  Springer, 2009.

\bibitem[Datta(2002)]{Datta/2002/Computing}
Ruchira Datta.
\newblock Computing {H}andelman representations.
\newblock In \emph{Mathematical Theory of Networks and Systems}, 2002.
\newblock Cf. { \url{math.berkeley.edu/~datta/MTNSHandelman.ps}}.

\bibitem[Dolzmann and Sturm(1997)]{Dolzmann+Sturm/97/REDLOG}
Andreas Dolzmann and Thomas Sturm.
\newblock {REDLOG}: Computer algebra meets computer logic.
\newblock \emph{ACM SIGSAM Bull.}, 31\penalty0 (2):\penalty0 2--9, June 1997.

\bibitem[Farouki(2012)]{Farouki/2012/Bernstein}
Rida~T. Farouki.
\newblock The {B}ernstein polynomial basis: A centennial retrospective.
\newblock \emph{Comput. Aided Geom. Des.}, 29\penalty0 (6):\penalty0 379--419,
  August 2012.

\bibitem[Forsman(1991)]{Forsman91constructionof}
K.~Forsman.
\newblock Construction of {L}yapunov functions using {G}r{\"o}bner bases.
\newblock In \emph{In Proc. of the 30th Conf. on Decision and Control}, pages
  798--799. IEEE, 1991.

\bibitem[Garey and Johnson(1979)]{Garey+Johnson/1979/Computers}
Michael~R. Garey and David~S. Johnson.
\newblock \emph{Computers and Intractability: A guide to the theory of
  {NP}-Completeness}.
\newblock W.H.Freeman, 1979.

\bibitem[Garloff(1993)]{Garloff93}
J.~Garloff.
\newblock The {B}ernstein algorithm.
\newblock \emph{Reliable Computing}, 2:\penalty0 154--168, 1993.

\bibitem[Goubault et~al.(2014)Goubault, Jourdan, Putot, and
  Sankaranarayanan]{Goubault+Others/2014/Finding}
Eric Goubault, Jacques-Henri Jourdan, Sylvie Putot, and Sriram
  Sankaranarayanan.
\newblock Finding non-polynomial positive invariants and lyapunov functions for
  polynomial systems through darboux polynomials.
\newblock In \emph{Proc. American Control Conference (ACC)}, pages 3571 --
  3578. IEEE Press, 2014.

\bibitem[Hafstein(2002)]{Hafstein2002}
S.~Hafstein.
\newblock \emph{Stability Analysis of Nonlinear Systems with Linear
  Programming}.
\newblock PhD thesis, Gerhard-Mercator-University Duisburg, 2002.

\bibitem[Hafstein(2014)]{Hafstein2014}
S.~Hafstein.
\newblock Revised {CPA} method to compute {L}yapunov functions for nonlinear
  systems.
\newblock \emph{Journal of Mathematical Analysis and Applications}, 4\penalty0
  (20):\penalty0 610--640, 2014.

\bibitem[Handelman(1988)]{Handelman/1988/Representing}
David Handelman.
\newblock Representing polynomials by positive linear functions on compact
  convex polyhedra.
\newblock \emph{Pacific J. Math}, 132\penalty0 (1):\penalty0 35--62, 1988.

\bibitem[Harrison(2007)]{Harrison/2007/Verifying}
John Harrison.
\newblock Verifying nonlinear real formulas via sums of squares.
\newblock In Klaus Schneider and Jens Brandt, editors, \emph{Proc. Intl. Conf.
  on Theorem Proving in Higher Order Logics}, volume 4732 of \emph{Lecture
  Notes in Computer Science}, pages 102--118. Springer-Verlag, 2007.

\bibitem[H{\"a}rter et~al.(2012)H{\"a}rter, Jansson, and Lange]{Jannson/VSDP}
V.~H{\"a}rter, C.~Jansson, and M.~Lange.
\newblock {VSDP}: A matlab toolbox for verified semidefinte-quadratic-linear
  programming, 2012.
\newblock Cf. {\tiny \url{http://www.ti3.tuhh.de/jansson/vsdp/}}.

\bibitem[Hausdorff(1921)]{Hausdorff1921}
F.~Hausdorff.
\newblock Summationsmethoden und {M}omentfolgen i.
\newblock \emph{Math. Zeit}, 9:\penalty0 74--109, 1921.

\bibitem[Henrion and Lasserre(2006)]{Henrion+Lasserre/2006}
D.~Henrion and J.B. Lasserre.
\newblock Convergent relaxations of polynomial matrix inequalities and static
  output feedback.
\newblock \emph{IEEE Transactions on Automatic Control}, 51\penalty0
  (42):\penalty0 192--202, 2006.

\bibitem[Jarvis-Wloszek(2003)]{jarvis2003}
Z.~W. Jarvis-Wloszek.
\newblock \emph{{L}yapunov Based Analysis and Controller Synthesis for
  Polynomial Systems using Sum-of-Squares Optimization}.
\newblock PhD thesis, UC Berkeley, 2003.

\bibitem[Johansen(2000)]{Johansen2000}
A.~Johansen.
\newblock Computation of {L}yapunov functions for smooth nonlinear systems
  using convex optimization.
\newblock \emph{Automatica}, 36\penalty0 (11):\penalty0 1617--1626, 2000.

\bibitem[Jones et~al.(2007)Jones, Baric, and
  Morari]{Jones+Others/2007/Multiparametric}
C.N. Jones, M.~Baric, and M.~Morari.
\newblock {Multiparametric Linear Programming with Applications to Control}.
\newblock \emph{European Journal of Control}, 13\penalty0 (2-3):\penalty0
  152--170, March 2007.
\newblock URL
  \url{http://control.ee.ethz.ch/index.cgi?page=publications;action=details;id=2699}.

\bibitem[Kamyar and Peet(2014)]{Kamyar+Peet/2014/Polynomial}
Reza Kamyar and Matthew~M. Peet.
\newblock Polynomial optimization with applications to stability analysis and
  control - alternatives to sum of squares.
\newblock \emph{arXiv}, abs/1408.5119, 2014.
\newblock Available online: \url{http://arxiv.org/abs/1408.5119}.

\bibitem[Kapinski et~al.(2014)Kapinski, Deshmukh, Sankaranarayanan, and
  Arechiga]{Kapinski+Others/2014/Simulation}
James Kapinski, Jyotirmoy~V. Deshmukh, Sriram Sankaranarayanan, and Nikos
  Arechiga.
\newblock Simulation-guided lyapunov analysis for hybrid dynamical systems.
\newblock In \emph{Hybrid Systems: Computation and Control (HSCC)}, pages
  133--142. ACM Press, 2014.

\bibitem[Kvasnica et~al.(2004)Kvasnica, Grieder, Baotic, and
  Morari]{Kvasnica+Others/2004/Multi}
M.~Kvasnica, P.~Grieder, M.~Baotic, and M.~Morari.
\newblock {Multi-Parametric Toolbox (MPT)}.
\newblock In \emph{HSCC (Hybrid Systems: Computation and Control )}, pages
  448--462, March 2004.
\newblock URL
  \url{http://control.ee.ethz.ch/index.cgi?page=publications;action=details;id=53}.

\bibitem[Lasserre(2001)]{Lasserre/2001/Global}
Jean~B. Lasserre.
\newblock Global optimization with polynomials and the problem of moments.
\newblock \emph{SIAM Journal on Optimization}, 11:\penalty0 796--817, 2001.

\bibitem[Lin and Rokne(1996)]{Rokne1995}
Q.~Lin and J.~Rokne.
\newblock Interval approximation of higher order to the ranges of functions.
\newblock \emph{Computers Math. Applic}, 31\penalty0 (7):\penalty0 101--109,
  1996.

\bibitem[Majumdar et~al.(2014)Majumdar, Ahmadi, and
  Tedrake]{Majumdar+Ahmadi+Tedrake/2014/Control}
A.~Majumdar, A.~A. Ahmadi, and R.~Tedrake.
\newblock Control and verification of high-dimensional systems via dsos and
  sdsos optimization.
\newblock In \emph{{IEEE} Conference on Decision and Control ({CDC})}, December
  2014.
\newblock To Appear (Dec. 2014).

\bibitem[Meiss(2007)]{Meiss/2007/Differential}
James~D. Meiss.
\newblock \emph{Differential Dynamical Systems}.
\newblock SIAM, 2007.

\bibitem[Monniaux and Corbineau(2011)]{Monniaux+Corbineau/2011/Generation}
David Monniaux and Pierre Corbineau.
\newblock On the generation of {P}ositivstellensatz witnesses in degenerate
  cases.
\newblock In \emph{ITP}, volume 6898 of \emph{Lecture Notes in Computer
  Science}, pages 249--264. Springer, 2011.

\bibitem[Motzkin(1967)]{Motzkin}
T.S. Motzkin.
\newblock The arithmetic-geometric inequality.
\newblock In \emph{Proc. Symposium on Inequalities}, pages 205--224. Acaemic
  Press, 1967.

\bibitem[Mu{\~{n}}oz and Narkawicz(2013)]{Munoz+Narkavicz/2013/Formalization}
C{\'{e}}sar Mu{\~{n}}oz and Anthony Narkawicz.
\newblock Formalization of a representation of {B}ernstein polynomials and
  applications to global optimization.
\newblock \emph{Journal of Automated Reasoning}, 51\penalty0 (2):\penalty0
  151--196, August 2013.
\newblock \doi{10.1007/s10817-012-9256-3}.
\newblock URL \url{http://dx.doi.org/10.1007/s10817-012-9256-3}.

\bibitem[Papachristodoulou et~al.(2013)Papachristodoulou, Anderson, Valmorbida,
  Prajna, Seiler, and Parrilo]{sostools}
A.~Papachristodoulou, J.~Anderson, G.~Valmorbida, S.~Prajna, P.~Seiler, and
  P.~A. Parrilo.
\newblock \emph{{SOSTOOLS}: Sum of squares optimization toolbox for {MATLAB}
  Version 3.00}, October 2013.

\bibitem[Papachristodoulou and
  Prajna(2002)]{Papachristodoulou+Prajna/2002/Construction}
Antonis Papachristodoulou and Stephen Prajna.
\newblock On the construction of {L}yapunov functions using the sum of squares
  decomposition.
\newblock In \emph{{IEEE} {CDC}}, pages 3482--3487. {IEEE} Press, 2002.

\bibitem[Papachristodoulou and
  Prajna(2005)]{Papachristodoulou+Prajna/2005/Analysis}
Antonis Papachristodoulou and Stephen Prajna.
\newblock Analysis of non-polynomial systems using the sum of squares
  decomposition.
\newblock In Didier Henrion and Andrea Garulli, editors, \emph{Positive
  Polynomials in Control}, volume 312 of \emph{Lecture Notes in Control and
  Information Science}, pages 23--43. Springer Berlin Heidelberg, 2005.
\newblock \doi{10.1007/10997703_2}.

\bibitem[Parillo(2003)]{Parillo/2003/Semidefinite}
Pablo~A Parillo.
\newblock Semidefinite programming relaxation for semialgebraic problems.
\newblock \emph{Mathematical Programming Ser. B}, 96\penalty0 (2):\penalty0
  293--320, 2003.

\bibitem[Platzer et~al.(2009)Platzer, Quesel, and
  R\"{u}mmer]{Platzer+Quesel+Rummer/2009/Real}
Andr{\'e} Platzer, Jan-David Quesel, and Philipp R\"{u}mmer.
\newblock Real world verification.
\newblock In \emph{Proceedings of Intl. Conf. on Automated Deduction}, pages
  485--501. Springer, 2009.

\bibitem[Podelski and Rybalchenko(2004)]{Podelski04}
A.~Podelski and A.~Rybalchenko.
\newblock A complete method for the synthesis of linear ranking functions.
\newblock \emph{Lecture Notes in Computer Science}, 2937:\penalty0 239--251,
  2004.

\bibitem[Powers and Reznick(2000)]{Powers2000}
Victoria Powers and Bruce Reznick.
\newblock Polynomials that are positive on an interval.
\newblock \emph{Trans. Amer. Maths. Soc}, 352:\penalty0 4677--4692, 2000.

\bibitem[Putinar(1993)]{Putinar1993}
M.~Putinar.
\newblock Positive polynomials on compact semi-algebraic sets.
\newblock \emph{Indiana Univ. Math}, 42:\penalty0 969--984, 1993.

\bibitem[Ratschan and She(2010)]{Ratschan+She/2010/Providing}
Stefan Ratschan and Zhikun She.
\newblock Providing a basin of attraction to a target region of polynomial
  systems by computation of {L}yapunov-like functions.
\newblock \emph{SIAM J. Control and Optimization}, 48\penalty0 (7):\penalty0
  4377--4394, 2010.

\bibitem[Sankaranarayanan et~al.(2013)Sankaranarayanan, Chen, and
  {\'A}braham]{Sriram2013}
Sriram Sankaranarayanan, Xin Chen, and Erika {\'A}braham.
\newblock {L}yapunov function synthesis using {H}andelman representations.
\newblock \emph{IFAC conference on Nonlinear Control Systems}, 2013.

\bibitem[Sassi et~al.(2012)Sassi, Testylier, Dang, and
  Girard]{Bensassi/reach/2012}
M.A.~Ben Sassi, R.~Testylier, T.~Dang, and A.~Girard.
\newblock Reachability analysis for polynomial system using linear programming
  relaxations.
\newblock In \emph{ATVA'2012}, pages 137--151, 2012.

\bibitem[Sassi et~al.(2014)Sassi, Sankaranarayanan, Chen, and
  Abraham]{BenSassi+Others/2014/Synthesis}
Mohamed Amin~Ben Sassi, Sriram Sankaranarayanan, Xin Chen, and Erika Abraham.
\newblock Linear relaxations of polynomial positivity for polynomial lyapunov
  function synthesis.
\newblock \emph{arXiv}, arXiv:1407.2952 [math.DS], 2014.

\bibitem[Schm{\"u}dgen(1991)]{Schmudgen1991}
K.~Schm{\"u}dgen.
\newblock The k-moment problem for compact semi-algebraic sets.
\newblock \emph{Math. Ann}, 289:\penalty0 203--206, 1991.

\bibitem[She et~al.(2009)She, Xiab, Xiaob, and Zhenga]{She2009}
Zhikun She, Bican Xiab, Rong Xiaob, and Zhiming Zhenga.
\newblock A semi-algebraic approach for asymptotic stability analysis.
\newblock \emph{Nonlinear Analysis: Hybrid Systems}, 3\penalty0 (4):\penalty0
  588--596, 2009.

\bibitem[She et~al.(2013)She, Li, Xue, Zhenga, and Xiab]{She2013}
Zhikun She, Haoyang Li, Bai Xue, Zhiming Zhenga, and Bican Xiab.
\newblock Discovering polynomial {L}yapunov functions for continuous dynamical
  systems.
\newblock \emph{Journal of Symbolic Computation}, 58:\penalty0 41--63, 2013.

\bibitem[Sherali and Tuncbilek(1991)]{sherali91}
H.D. Sherali and C.H. Tuncbilek.
\newblock A global optimization algorithm for polynomial programming using a
  reformulation-linearization technique.
\newblock \emph{Journal of Global Optimization}, 2:\penalty0 101--112, 1991.

\bibitem[Sherali and Tuncbilek(1997)]{sherali97}
H.D. Sherali and C.H. Tuncbilek.
\newblock New reformulation-linearization/convexification relaxations for
  univariate and multivariate polynomial programming problems.
\newblock \emph{Operation Research Letters}, 21:\penalty0 1--9, 1997.

\bibitem[Shor(1987)]{Shor/1987/Class}
N.Z. Shor.
\newblock Class of global minimum bounds on polynomial functions.
\newblock \emph{Cybernetics}, 23\penalty0 (6):\penalty0 731--734, 1987.
\newblock Originally in Russian: Kibernetika (6), 1987, 9--11.

\bibitem[Tabuada(2009)]{Tabuada/09/Verification}
Paulo Tabuada.
\newblock \emph{Verification and Control of Hybrid Systems: A Symbolic
  Approach}.
\newblock Springer, 2009.

\bibitem[Tan and Packard(2007)]{Tan+Packard/2006}
W.~Tan and A.~Packard.
\newblock Stability region analysis using sum of squares programming.
\newblock In \emph{Proc. {ACC}}, 2007.

\bibitem[Tarski(1951)]{Tarski/51/Decision}
Alfred Tarski.
\newblock A decision method for elementary algebra and geometry.
\newblock Technical report, Univ. of California Press, Berkeley, 1951.

\bibitem[Tibken(2000)]{Tibken/2000/Estimation}
B.~Tibken.
\newblock Estimation of the domain of attraction for polynomial systems via
  {LMI}s.
\newblock In \emph{{IEEE} {CDC}}, volume~4, pages 3860--3864 vol.4. {IEEE}
  Press, 2000.

\bibitem[Topcu et~al.(2007)Topcu, Packard, Seiler, and
  Wheeler]{Topcu+Packard+Seiler+Wheeler/2007/Stability}
Ufuk Topcu, Andrew Packard, Peter Seiler, and Timothy Wheeler.
\newblock Stability region analysis using simulations and sum-of-squares
  programming.
\newblock In \emph{Proc. {ACC}}, pages 6009--6014. {IEEE} Press, 2007.

\bibitem[Vannelli and Vidyasagar(1985)]{Vannelli+Vidyasagar/1985/Maximal}
A.~Vannelli and M.~Vidyasagar.
\newblock Maximal {L}yapunov functions and domains of attraction for autonomous
  nonlinear systems.
\newblock \emph{Automatica}, 21\penalty0 (1):\penalty0 69--80, 1985.
\newblock ISSN 0005-1098.
\newblock \doi{http://dx.doi.org/10.1016/0005-1098(85)90099-8}.

\end{thebibliography}

\iftrversion
\appendix
\section{Description of Synthesized Benchmarks}\label{App:description}
In this section,we describe each of the $15$ benchmarks and present
the results of our implementation.

\paragraph{Benchmark \#1:} Consider the two variable polynomial ODE:
\begin{align*}
\frac{dx}{dt} &=-12.5x+2.5x^2+2.5y^2+10x^2y+2.5y^3.\\
\frac{dy}{dt} & =-y-y^2.\\
\end{align*}

The second relaxation finds the Lyapunov function and derivative shown below:
\begin{lstlisting}
Lyapunov function :

  2.02x^2+5y^2.  

Lyapunov derivative function :

-50.5x^2-10y^2+10.1x^3+10.1xy^2-10y^3+40.4x^3y+10.1xy^3.
   

\end{lstlisting}
\paragraph{Benchmark \#2:} Consider the two variable polynomial ODE:
\begin{align*}
\frac{dx}{dt} &= 6.933333x^3+4.566667x^2-21.5x.\\
\frac{dy}{dt} & = 6.933333x^3+0.4x^2y+2.066667x^2+xy^2+0.6xy-9x-y^2-y.\\
\end{align*}
The first relaxation finds the Lyapunov function and derivative shown below:
\begin{lstlisting}
Lyapunov function :

  4.9183x^2-3.3198xy+4.1497y^2.  

Lyapunov derivative function :

  -181.6089x^2-8.2995y^2+38.0596x^3+8.2995xy^2-8.2995y^3. 
  +45.1833x^4+33.1978x^3y+8.2995xy^3.

\end{lstlisting}
\paragraph{Benchmark \#3:} Consider the two variable polynomial ODE:
\begin{align*}
\frac{dx}{dt} &=-1.5x-x^2+0.5xy+0.5y^2-2x^3+x^2y.\\
\frac{dy}{dt} & =-0.5y.\\ 
\end{align*}
The first relaxation finds the Lyapunov function and derivative shown below:
\begin{lstlisting}
Lyapunov function :

  4.9693x^2+4.8581y^2.  
    
Lyapunov derivative function :

  -14.908x^2-4.8581y^2-9.9386x^3+4.9693x^2y+4.9693xy^2-19.8773x^4+9.9386x^3y.

  \end{lstlisting}
\paragraph{Benchmark \#4:} Consider the two variable polynomial ODE:
\begin{align*}
\frac{dx}{dt} &=-2x^3-0.5xy-0.5x.\\
\frac{dy}{dt} & =0.25xy^2-0.125xy+0.25y^2-0.4125y.\\ 
\end{align*}

The first relaxation finds the Lyapunov function and derivative shown below:
\begin{lstlisting}
Lyapunov function :

  4.9663x^2+4.8552y^2.
    
Lyapunov derivative function :

  -4.9663x^2-4.0056y^2-4.9663x^2y-1.2138xy^2+2.4276y^3-19.8653x^4+2.4276xy^3.


  \end{lstlisting}
 \paragraph{Benchmark \#5:} Consider the three variable polynomial ODE:
\begin{align*}
\frac{dx}{dt} &=-2x^3-0.5xy-0.5x-z^3-z^2.\\
\frac{dy}{dt} & =0.25xy^2-0.125xy+0.25y^2-0.4125y.\\ 
\frac{dz}{dt} & =-z^2-z.\\
\end{align*}

The first relaxation finds the Lyapunov function and derivative shown below:
\begin{lstlisting}
 Lyapunov function :

   4.9295x^2+4.9513y^2+4.9848z^2.  

Lyapunov derivative function :

  -4.9295x^2-4.0848y^2-9.9696z^2-4.9295x^2y-1.2378xy^2-9.859xz^2
  +2.4756y^3-9.9696z^3-19.7179x^4+2.4756xy^3-9.859xz^3.

  \end{lstlisting}
  
 \paragraph{Benchmark \#6:} Consider the three variable polynomial ODE:
\begin{align*}
\frac{dx}{dt} &=-0.5x^3y+0.5x^3z^2-3x^3+y^5-y^4+yz^4-z^4.\\
\frac{dy}{dt} & =0.25y^2-0.25y.\\ 
\frac{dz}{dt} & =yz^4+z^4-2z^3.\\
\end{align*}

The third relaxation finds the Lyapunov function and derivative shown below:
\begin{lstlisting}
 Lyapunov function :

   4.1212x^4-0.0292x^3y+4.9077x^3+3.5749x^2y^2+4.9755x^2z^2  
  +4.9863x^2-1.5913xy^3+1.5914xy^2+4.9939y^4+0.0598y^3z    
  -1.1362y^3+4.9812y^2z^2-0.0597y^2z+4.9950y^2+0.0198yz^3  
  +4.9864z^4+4.9926z^2  
     
 Lyapunov derivative function :

  -2.4975y^2-0.79571xy^2+3.3496y^3+0.029844y^2z-29.9178x^4-1.7881x^2y^2
  +1.9892xy^3-5.8461y^4-0.07469y^3z-2.4885y^2z^2-0.0 049472yz^3-19.9701z^4
  -44.1693x^5-4.9696x^4y+0.041946x^4z-4.7815x^3y^2+0.013348x^3z^2+1.7874x^2y^3
  +0.013982x^2z^3-11.166xy^4-9.9552xz^4+4.994y^5+0.0452y^4+2.4906y^3z^2
  +0.12432y^2z^3-0.033711yz^4+9.9792z^5-49.4547x^6-7.0989x^5y-0.057848x^5z
  -21.4464x^4y^2-24.8666x^4z^2+3.9781x^3y^3-0.011402x^3yz^2-14.7231x^2y^4
  -34.6321x^2z^4+9.9782xy^5+0.013982xy^4z+9.9597xyz^4+0.01905xz^5-1.5914y^6
  -0.11959y^3z^3-21.576y^2z^4+9.8832yz^5-39.8832z^6-8.2424x^6y+0.04377x^5y^2
  +7.3615x^5z^2-3.575x^4y^3-4.9783x^4yz^2-15.6893x^3y^4+0.79344x^3y^2z^2
  -16.4807x^3z^4+14.8106x^2y^5-0.019283x^2y^4+14.8042x^2yz^4+9.9317x^2z^5
  -7.1553xy^6-0.01503xy^5z-9.951xy^4z^2-7.155xy^2z^4-0.0148xyz^5-9.958xz^6
  +3.1827y^7+3.1828y^3z^4+9.9793y^2z^5+0.053588yz^6+19.9477z^7+8.2424x^6z^2
  -0.043771x^5yz^2+3.5749x^4y^2z^2+4.975x^4z^4+16.485x^3y^5-0.79563x^3y^3z^2
  +16.4936x^3yz^4-0.08754x^2y^6+0.019x^2y^5z-0.087017x^2y^2z^4+9.9703x^2yz^5
  +7.1497xy^7+9.951xy^5z^2+7.15xy^3z^4+0.0101xy^2z^5+9.944xyz^6-1.5913y^8
  -1.5315y^4z^4+9.9627y^3z^5+0.063907y^2z^6+19.9431yz^7.
 
  \end{lstlisting}
  \paragraph{Benchmark \#7:} Consider the three variable polynomial ODE:
\begin{align*}
\frac{dx}{dt} &=-0.5x^3y+0.5x^3z^2-x^3+y^4z+y^4-yz^3+yz^2+z^3-z^2\\
\frac{dy}{dt} & = 0.5y^2z-0.5y^2-2y\\ 
\frac{dz}{dt} & =-yz^2+yz+z^2-z\\
\end{align*}

The first relaxation finds the Lyapunov function and derivative shown below:
\begin{lstlisting}
Lyapunov function :

   1.8371x^5+0.1146x^4y+0.1431x^4z+4.9587x^4-2.0557x^3y^2       
  -0.0014x^3yz-0.4698x^3y+3.2944x^3z^2+4.0441x^3-1.2295x^2y^3       
  +4.9584x^2y^2+3.2610x^2yz^2+0.6981x^2z^3+4.9648x^2z^2+4.9858x^2            
  +1.9598xy^4+0.9480xy^3+1.0295xy^2z^2+0.6737xyz^3+2.3539xyz^2  
  -1.1976xz^4+0.2212xz^3-3.3047xz^2-0.3773y^5+0.1262y^4z    
   +4.9884y^4-1.7272y^3z^2-4.5919y^3+0.7677y^2z^3+4.9842y^2z^2  
   +4.9898y^2+1.8746yz^4+0.4655yz^3+0.9830yz^2+0.8032z^5  
   +4.9823z^4-0.6791z^3+4.9962z^2  

Lyapunov derivative function :

  -10.087xz^2+24.0416yz^2+16.1044z^3-9.9716x^4-66.1853x^2z^2
  +15.7876xyz^2+1.3715xz^3-19.9593y^4-71.2314y^2z^2-12.204yz^3
  -48.9619z^4-12.1323x^5-4.986x^4y-0.42935x^4z-77.376x^3z^2
  +50.8098x^2yz^2+15.7832x^2z^3+9.9739xy^4-27.3951xy^2z^2
  -11.3265xyz^3-7.6925xz^4+17.5716y^5-0.37814y^4z+65.4124y^3z^2
   +16.8366y^2z^3-8.169yz^4+4.3258z^5-19.8348x^6-4.6566x^5y
  -9.917x^4y^2-32.0142x^4z^2+45.0923x^3yz^2+24.9503x^3z^3
  -7.7016x^2y^4+3.208x^2y^2z^2-16.9243x^2yz^3-17.7736x^2z^4
  -5.6866xy^5+9.9714xy^4+55.9384xy^3z^2+9.1352xy^2z^3+18.4209xyz^4
  +0.61781xz^5-26.132y^6+9.9786y^5z-32.7928y^4z^2-21.8781y^3z^3
  -9.2688y^2z^4+6.4374yz^5+12.9884z^6-9.1857x^7-10.3757x^6y
  -0.57258x^6z+6.8718x^5y^2-3.8171x^5z^2-2.7287x^4y^3+16.0705x^4yz^2
  +7.6749x^4z^3+26.0935x^3y^4-0.83166x^3y^2z^2-14.3316x^3yz^3
   +0.1192x^3z^4-3.9495x^2y^5+12.1284x^2y^4z-25.0232x^2y^3z^2
  -1.0817x^2y^2z^3+29.6455x^2yz^4+6.6222x^2z^5-8.6057xy^6-3.9191xy^4z^2
  -21.5587xy^3z^3+19.5651xy^2z^4-1.2779xyz^5+0.72252xz^6-15.233y^7
  -14.7866y^6z+8.6278y^5z^2-2.7363y^4z^3-1.1648y^3z^4+5.8781y^2z^5
  -4.6757yz^6-3.0722z^7-4.5929x^7y-0.22917x^6y^2-0.28629x^6yz
  +9.9174x^6z^2+3.0835x^5y^3-5.6463x^5yz^2+10.3006x^4y^4
  +1.6974x^4y^2z^2-9.8837x^4yz^3+4.9648x^4z^4+3.5898x^3y^5
  +19.9386x^3y^4z-0.040754x^3y^3z^2-0.79519x^3y^2z^3+1.2032x^3yz^4
  +0.11062x^3z^5-2.4786x^2y^6+8.502x^2y^5z+6.6209x^2y^4z^2
  +6.167x^2y^3z^3-9.8832x^2yz^5-10.2983xy^7+12.7599xy^6z
  +4.4632xy^5z^2+15.465xy^4z^3-6.5221xy^2z^5-1.3962xyz^6
  +3.8462y^8+20.3974y^7+6.2115y^6z^2+9.5007y^5z^3-2.3862y^4z^4
  -1.0295y^3z^5-0.67367y^2z^6+1.1976yz^7+4.5929x^7z^2
  +0.22917x^6yz^2+0.28629x^6z^3-3.0835x^5y^2z^2+4.9416x^5z^4
  +9.3003x^4y^4z-1.2295x^4y^3z^2+3.261x^4yz^4+0.6981x^4z^5
  -3.653x^3y^5z+1.5511x^3y^4z^2+0.51474x^3y^2z^4+0.33683x^3yz^5
  -0.5988x^3z^6-9.8555x^2y^6z+13.1442x^2y^4z^3+5.3804xy^7z
  +8.581xy^5z^3+2.0699xy^4z^4+0.073499y^8z+0.50561y^7z^2
  -4.152y^6z^3+2.209y^5z^4+0.677y^4z^5.


  
 \end{lstlisting}
  \paragraph{Benchmark \#8:} Consider the three variable polynomial ODE:
\begin{align*}
\frac{dx}{dt} &=-0.5x^3y+0.5x^3z^2-x^3+y^4z+y^4-yz^3+3yz^2+z^3-3z^2\\
\frac{dy}{dt} & =y^4z-y^4-2y^3-z^3+3z^2\\ 
\frac{dz}{dt} & =z^2-3z\\
\end{align*}

The first relaxation finds the Lyapunov function and derivative shown below:
\begin{lstlisting}
Lyapunov function :

    1.8371x^5+0.1146x^4y+0.1431x^4z+4.9587x^4 -2.0557x^3y^2       
    -0.4698x^3y+3.2944x^3z^2+4.0441x^3 -1.2295x^2y^3+4.9584x^2y^2       
   +3.2610x^2yz^2+0.6981x^2z^3+4.9648x^2z^2+4.9858x^2+1.9598xy^4       
   +0.9480xy^3+1.0295xy^2z^2+0.6737xyz^3+2.3539xyz^2 -1.1976xz^4  
    0.2212xz^3 -3.3047xz^2   -0.3773y^5+0.1262y^4z+4.9884y^4       
    -1.7272y^3z^2 -4.5919y^3+0.7677y^2z^3+4.9842y^2z^2+4.9898y^2       
    +1.8746yz^4+0.4655yz^3+0.9830yz^20.8032z^5+4.9823z^4  
   -0.6791z^3+4.9962z^2  
     
Lyapunov derivative function :

   -29.977z^2-10.087xz^2+24.0416yz^2+16.1044z^3-9.9716x^4-66.1853x^2z^2
   +15.7876xyz^2+1.3715xz^3-19.9593y^4-71.2314y^2z^2-12.204yz^3-48.9619z^4
  -12.1323x^5-4.986x^4y-0.42935x^4z-77.376x^3z^2+50.8098x^2yz^2+15.7832x^2z^3
  +9.9739xy^4-27.3951xy^2z^2-11.3265xyz^3-7.6925xz^4+17.5716y^5-0.37814y^4z
   +65.4124y^3z^2+16.8366y^2z^3-8.169yz^4+4.3258z^5-19.8348x^6-4.6566x^5y 
   -9.917x^4y^2-32.0142x^4z^2+45.0923x^3yz^2+24.9503x^3z^3-7.7016x^2y^4
   +3.208x^2y^2z^2-16.9243x^2yz^3-17.7736x^2z^4-5.6866xy^5+9.9714xy^4z
   +55.9384xy^3z^2+9.1352xy^2z^3+18.4209xyz^4+0.61781xz^5-26.132y^6
   +9.9786y^5z-32.7928y^4z^2-21.8781y^3z^3-9.2688y^2z^4+6.4374yz^5
   +12.9884z^6-9.1857x^7-10.3757x^6y-0.57258x^6z+6.8718x^5y^2-3.8171x^5z^2
   -2.7287x^4y^3-0.0013497x^4y^2+16.0705x^4yz^2+7.6749x^4z^3+26.0935x^3y^4
   -0.83166x^3y^2z^2-14.3316x^3yz^3+0.1192x^3z^4-3.9495x^2y^5+12.1284x^2y^4z
   -25.0232x^2y^3z^2-1.0817x^2y^2z^3+29.6455x^2yz^4+6.6222x^2z^5-8.6057xy^6
   -3.9191xy^4z^2-21.5587xy^3z^3+19.5651xy^2z^4-1.2779xyz^5+0.72252xz^6
   -15.233y^7-14.7866y^6z+8.6278y^5z^2-2.7363y^4z^3-1.1648y^3z^4+5.8781y^2z^5
   -4.6757yz^6-3.0722z^7-4.5929x^7y-0.22917x^6y^2-0.28629x^6yz+9.9174x^6z^2
   +3.0835x^5y^3-5.6463x^5yz^2+10.3006x^4y^4+1.6974x^4y^2z^2-9.8837x^4yz^3
   +4.9648x^4z^4+3.5898x^3y^5+19.9386x^3y^4z-0.040754x^3y^3z^2-0.79519x^3y^2z^3
   +1.2032x^3yz^4+0.11062x^3z^5-2.4786x^2y^6+8.502x^2y^5z+6.6209x^2y^4z^2
   +6.167x^2y^3z^3-9.8832x^2yz^5-10.2983xy^7+12.7599xy^6z+4.4632xy^5z^2
   +15.465xy^4z^3-6.5221xy^2z^5-1.3962xyz^6+3.8462y^8+20.3974y^7z+6.2115y^6z^2
   +9.5007y^5z^3-2.3862y^4z^4-1.0295y^3z^5-0.67367y^2z^6
   +1.1976yz^7+4.5929x^7z^2+0.22917x^6yz^2+0.28629x^6z^3-3.0835x^5y^2z^2
   +4.9416x^5z^4+9.3003x^4y^4z-1.2295x^4y^3z^2+3.261x^4yz^4+0.6981x^4z^5-3.653x^3y^5z    
   +1.5511x^3y^4z^2+0.51474x^3y^2z^4+0.33683x^3yz^5-0.5988x^3z^6-9.8555x^2y^6z
   +13.1442x^2y^4z^3+5.3804xy^7z+8.581xy^5z^3+2.0699xy^4z^4+0.073499y^8z
   +0.50561y^7z^2-4.152y^6z^3+2.209y^5z^4+0.677y^4z^5.


  \end{lstlisting} 
   \paragraph{Benchmark \#9:} Consider the three variable polynomial ODE:
\begin{align*}
\frac{dx}{dt} &=  0.05x^2yz+0.05x^2y-0.05x^2z-0.05x^2+0.05xyz+0.05xy-0.05xz-0.05x+0.125y^3z-0.125y^3\\
&+0.125y^2z-0.125y^2+0.2yz^5+0.2yz^4-0.2z^5-0.2z^4; \\
\frac{dy}{dt} & = 0.125y^2z-0.125y^2+0.125yz-0.125y+0.2z^5+0.2z^4\\ 
\frac{dz}{dt} & =-0.1z^2-0.1z\\
\end{align*}

The second relaxation finds the Lyapunov function and derivative shown below:

Lyapunov function :

\[   2.7500x^2+2.7500y^2+5.0000z^2  \]

Lyapunov derivative function :
\begin{lstlisting}
  -0.275x^2-0.6875y^2-z^2-0.275x^3+0.275x^2y-0.275x^2z-0.6875xy^2-0.6875y^3
  +0.6875y^2z-z^3+0.275x^3y-0.275x^3z+0.275x^2yz-0.6875xy^3+0.6875xy^2z
  +0.6875y^3z+0.275x^3yz+0.6875xy^3z-1.1xz^4+1.1yz^4+1.1xyz^4-1.1xz^5+1.1yz^5+1.1xyz^5
\end{lstlisting}

   \paragraph{Benchmark \#10:} Consider the three variable polynomial ODE:
\begin{align*}
\frac{dx}{dt} &=-0.01x+1.666667xz^2y^2-1.111111xz^y+0.555556xz^2-0.555556z^2 \\
&-1.111111zy^3+1.111111zy^2+1.111111y^3-1.111111y^2  \\
\frac{dz}{dt} & =-5zy^2+5zy-7.5z-5y^3+5y^2\\ 
\frac{dy}{dt} & =2y^2-2y\\
\end{align*}

The third relaxation finds the Lyapunov function shown below:
\begin{lstlisting}
Lyapunov function :

  1.5308x^2+4.9266z^2+4.9819y^2  
   
Lyapunov derivative function :

  -0.030616x^2-73.8988z^2-19.9274y^2-1.7009xz^2-3.4017xy^2+49.2659z^2y
  +49.2659zy^2+19.9274y^3+1.7009x^2z^2+3.4017xzy^2+3.4017xy^3-49.2659z^2y^2-
  +49.2659zy^3-3.4017x^2z^2y-3.4017xzy^3+5.1026x^2z^2y^2
  \end{lstlisting}

 \paragraph{Benchmark \#11:} Consider the four variable polynomial ODE:
\begin{align*}
\frac{dx}{dt} &=-18xyw-13xy-18xw-37.5x-16z^3+4z^2y-31.5z^2w-6.5z^2+32zyw+48zy-16zw^2-36zw\\
&+8y^3+36y^2w+28y^2+68yw+16y-14w^2   \\
\frac{dz}{dt} & =-16z^2+24zy-31.5zw-27.5z-32y^2+32yw+16y-16w^2-28w\\ 
\frac{dy}{dt} & =-36y^2w-52y^2-36yw-112y+64w\\
\frac{dw}{dt} & =-4w.\\ 
\end{align*}
The first relaxation finds the Lyapunov function and derivative shown below:
\begin{lstlisting}
Lyapunov function :
    1.6209y^2 +1.3650 yw + 4.8875w^2     
Lyapunov derivative function :

   -363.0877y^2-303.641w^2-168.5765y^3-187.6862y^2w
    -49.1398yw^2-116.7068y^3w-49.1397y^2w^2

 \end{lstlisting} 
 \paragraph{Benchmark \#12:} Consider the four variable polynomial ODE:
\begin{align*}
\frac{dx}{dt} &= 28x^3-28x^2z-28x^2y+0.5x^2w+9.5x^2+3xz^2+28xzy-xzw+21xz+14xy^2+2xyw-1.5x+10.5xw-60.5x-6z^2y \\
&-15.5z^2w+19.5z^2-22.5zy^2-2zyw-18zy+9zw+9z+12.5y^3-8y^2w+8y^2+1yw^2-8yw+41y+12.5w^2+6w    \\
\frac{dz}{dt} & = 2z^3+4z^2y+8.5z^2w+4.5z^2+4zy^2+5.75zyw-7.25zy+8.5zw^2-11zw-42.5z+9y^2w+17.75y^2+22.5yw^2\\
&+12.5yw-23y+2.25w^3+11.25w^2-7w\\ 
\frac{dy}{dt} & =-21y^2-12yw-129y-45w^3-101w^2-62w\\
\frac{dw}{dt} & =-13.5w^2-27w.\\ 
\end{align*}
The second relaxation finds the weak Lyapunov function shown below:
\begin{lstlisting}

Lyapunov function :
 
  1.5759y^2-1.2527yw+5.0000w^2  
 
Lyapunov derivative function :

  -406.592y^2-192.3343w^2-66.1895y^3-11.5165y^2w
  -286.3966yw^2-8.4804w^3-141.8345yw^3+56.3701w^4

\end{lstlisting}

 \paragraph{Benchmark \#13:} Consider the four variable polynomial ODE:
 \begin{align*}
\frac{dx}{dt} &=-1.510417x^5+8x^4yw+8.5x^4y-8x^4w-12.208333x^4-12x^3zyw\\&-9.75x^3zy-6x^3zw+2x^3y^2w+22x^3y^2+4x^3yw+6.5x^3y+2.5x^3w^2-47x^3w-60.875x^3\\ &-8x^2z^3w+2x^2z^2yw-16.875x^2z^2y+8x^2z^2w^2-13x^2z^2w-8x^2zy^2w \\&-7.5x^2zy^2+2x^2zyw^2+37x^2zyw-3.75x^2zy-4x^2zw^3-14.75x^2zw^2-46.5x^2zw-8x^2y^3w\\ &-7.5x^2y^3+4x^2y^2w+1x^2y^2+16x^2yw^3+6.5x^2yw^2-2x^2yw+2x^2y-12x^2w^3\\&+6.5x^2w^2-7x^2w+11.75x^2-4xz^4w-7xz^3yw-6.4375xz^3y+16xz^3w^2+25.5xz^3w\\ 
&+4xz^2y^2w+12.25xz^2y^2-2xz^2yw^2+26.5xz^2yw+1.125xz^2y-1xz^2w^3 \\&-60.875xz^2w^2-47.75xz^2w+44xzy^3w+54.25xzy^3-24xzy^2w^2-83xzy^2w-55.5xzy^2\\ 
&+49.25xzyw^2+29xzyw-13xzy+42xzw^3-20.75xzw^2-32.5xzw-1.5xy^4+16xy^3w^2\\&+9xy^3w-0.5xy^3-29.5xy^2w^2-43xy^2w-45.5xy^2+16xyw^3+30.5xyw^2+15xyw+24xy\\ &-6xw^3-64.5xw^2+58.5xw-41.833333x-4z^5w+6.5z^4yw-12.71875z^4y+12z^4w^2 \\&-7.25z^4w-6z^3y^2w-8.375z^3y^2-9z^3yw^2-15.75z^3yw-22.4375z^3y-9z^3w^3\\ 
&-50.4375z^3w^2-69.875z^3w+14z^2y^3w+11.625z^2y^3-2z^2y^2w^2-56.5z^2y^2w\\&-34.75z^2y^2+54.625z^2yw^2-33.5z^2yw+11z^2y+61z^2w^3+14.625z^2w^2-0.25z^2w\\
&-8zy^4w-20.75zy^4+8zy^3w^2+18.5zy^3w+22.75zy^3-31.75zy^2w^2-50.5zy^2w\\&-17.75zy^2-12zyw^3+1.25zyw^2+71.5zyw+8zy-1zw^4+33zw^3-143.25zw^2-1.75zw\\
 &+16y^5w+18y^5-16y^4w^2-56y^4w-46y^4+5y^3w^2+4y^3w+22y^3+8y^2w^4\\&+1y^2w^3-49y^2w^2-137y^2w+19y^2+2yw^5-25.5yw^4-9yw^3+31.5yw^2-55yw+12y\\
 &-2w^5-23.5w^4-11w^3+31.5w^2-11w \\
\frac{dz}{dt} & =-3.020833x^5+15.46875x^4z+10.583333x^4-21.0625x^3z^2-18.625x^3z\\&-31.75x^3+18.875x^2z^3+9.75x^2z^2-51.625x^2z+8.5x^2-10.25xz^4-7.5xz^3\\
&+12.75xz^2+14.25xz-40.666667x+3.5z^5-8z^4-51.5z^3-5.5z^2-41.5z\\ 
\frac{dy}{dt} & =3.25z^5w-6.359375z^5+13z^4yw+15.3125z^4y-4.5z^4w^2+0.125z^4w\\&-44.71875z^4-9z^3y^2w-10.1875z^3y^2-1z^3yw^2-20.25z^3yw-20.875z^3y\\ 
&+42.3125z^3w^2-12.75z^3w-3.5z^3-6.375z^2y^3+4z^2y^2w^2-6.75z^2y^2w\\&-17.125z^2y^2+27.625z^2yw^2-9.25z^2yw-72.375z^2y-6z^2w^3-19.375z^2w^2-23.25z^2w\\
&-3.5z^2+8zy^4w+9zy^4-8zy^3w^2-2zy^3w-23zy^3+9zy^2w^2-42zy^2w+38zy^2\\&+4zyw^4-11zyw^3-57.5zyw^2-164.5zyw+45.5zy+zw^5-12.75zw^4-44.5zw^3+74.75zw^2\\
&-23.5zw-22z+4y^5w+4y^5+8y^4w+6y^4-16y^3w^2-16y^3w-85y^3+8y^2w^3\\&+3y^2w^2+11y^2w-1y^2+8yw^4+6yw^3-52.5yw^2+17yw-25.5y-8w^5-16.5w^4-3w^3-14.5w^2+9w \\
\frac{dw}{dt} & =-2z^6+6z^5w+5.375z^5-4.5z^4w^2-22.21875z^4w-74.9375z^4\\&+34.5z^3w^2+22.3125z^3w-15.125z^3-5z^2w^3+33.5z^2w^2-128.125z^2w \\
&-40.875z^2-zw^4-21.75zw^3+13.5zw^2+13.75zw+4.5z-12w^4-22w^3-12.5w^2-4w \\ 
\end{align*}

 \paragraph{Benchmark \#14:} Consider the four variable polynomial ODE:
\begin{align*}
\frac{dx}{dt} &=7.145833x^5-20x^4y-2.416667x^4-10x^3zy+16x^3zw+20x^3y^2-18x^3yw-28x^3y\\&-10x^3w^2-12x^3w-77.541667x^3+3.5x^2z^2yw-25x^2z^2y+2.5x^2z^2w\\
&-20x^2zy^2w-30x^2zy^2+12x^2zyw^2-9x^2zyw+11x^2zy-12x^2zw^3-21x^2zw^2+15x^2zw\\&+28x^2y^3w+28x^2y^3+28x^2y^2w^2+26x^2y^2w+18x^2y^2+14x^2yw^3+2x^2yw^2\\
&-40x^2yw-7x^2y-2x^2w^3-8x^2w^2-17x^2w-42x^2+13.75xz^3yw+5.5xz^3y-24xz^3w^2\\&-2.75xz^3w+4xz^2y^2w+9xz^2y^2+32xz^2yw^2-2xz^2yw+9.5xz^2y-6xz^2w^3\\ 
&-10.5xz^2w^2+31.5xz^2w+2xzy^3w+19xzy^3-24xzy^2w^2-12xzy^2w+2xzy^2+31xzyw^3\\&+43xzyw^2+15xzyw-23.5xzy-13xzw^3-14xzw^2-70.5xzw+2xy^4+28xy^3w^2\\
&+7xy^3w+45xy^3+38xy^2w^3+83xy^2w^2-16.5xy^2w-142xy^2+28xyw^4-12xyw^3 \\&-76xyw^2-34xyw+25xy-23xw^4-38xw^3-64.5xw^2+48xw-92x-12z^5w+10.375z^4yw\\
&+4.25z^4y-12z^4w^2+3.125z^4w-1z^3y^2w+0.5z^3y^2+4z^3yw^2+16z^3yw+23.75z^3y\\&-15z^3w^3-30.25z^3w^2+25.75z^3w+1z^2y^3w+39.5z^2y^3-32z^2y^2w^2\\
&+16z^2y^2w+63z^2y^2+29.5z^2yw^3-40.5z^2yw+47.75z^2y-36.5z^2w^3-27.5z^2w^2\\&-88.25z^2w-3zy^4w+7zy^4+26zy^3w^2+48zy^3w+26.5zy^3+7zy^2w^3+35.5zy^2w^2\\ &+26.75zy^2w-51zy^2+14zyw^4-2zyw^3-5zyw^2-62zyw-50.5zy-19.5zw^4-29zw^3 \\&-53.25zw^2-37zw+2.5y^5-2y^4w-8.5y^4+24y^3w^3+72.5y^3w^2+88y^3w+3y^3+18y^2w^3\\
&-94.5y^2w^2-126y^2w-58y^2+14yw^5+5yw^4+10yw^3-46.5yw^2+84yw+12y\\&-14w^5-42w^4+16w^3+29.5w^2-13w\\
\frac{dz}{dt} & = 4.291667x^5+5.4375x^4z-20.833333x^4-11.125x^3z^2+12.75x^3z-102.083333x^3-10.25x^2z^3-0.5x^2z^2-2.625x^2z\\&-88x^2-24.5xz^4-21xz^3-51.25xz^2-70xz-57x+5z^5+7z^4-112.5z^3-31z^2-90z\\
\frac{dy}{dt} & =5.1875z^5w-7.875z^5+5z^4yw+10.25z^4y+20z^4w^2-4.5z^4w-2.625z^4\\&+5z^3y^2w+9.75z^3y^2-16z^3yw^2-12z^3yw+9.5z^3y+14.75z^3w^3-56z^3w^2\\ &-85.25z^3w-33.125z^3+z^2y^3w+31z^2y^3+3z^2y^2w^2+8z^2y^2w+7.25z^2y^2\\&+3.5z^2yw^3-14.25z^2yw^2-34.625z^2yw-74z^2y+7z^2w^4+4z^2w^3+24.5z^2w^2\\ &-20z^2w-35.75z^2+10zy^4w+11.25zy^4-10zy^3w^2+22zy^3w+28.75zy^3+12zy^2w^3\\&+4.25zy^2w^2+44zy^2w-10.5zy^2+37zyw^3+24.75zyw^2-83zyw-58zy+7zw^5\\ &+41zw^4+30zw^3+47.75zw^2-12zw-18z+10y^5w+10y^5-20y^4w^2+17.5y^4\\&+12.5y^3w^3-7.5y^3w^2+8y^3w-100y^3+10y^2w^3+103.5y^2w^2+14y^2w-4y^2-42.5yw^3\\
&-70.5yw^2-15yw-36.5y+10w^5-8w^4+20w^3+8.5w^2-12w\\ 
\frac{dw}{dt} & =-6z^6-6z^5w+13.5625z^5-7.5z^4w^2-15.125z^4w+3.375z^4-6.25z^3w^2\\&-22.75z^3w+9.875z^3-9.75z^2w^3-38.5z^2w^2-74.125z^2w-91z^2-7zw^4\\&-13zw^3+8zw^2+21.75zw+30.5z-13.5w^4-39.5w^3-7w^2-43w\\ 
\end{align*}

\paragraph{Benchmark \#15:} Consider the four variable polynomial ODE:
 \begin{align*}
\frac{dx}{dt} &= 8x^5+84x^4zw-2x^4z-204.5x^4yw-4x^4w^2+42x^4w+29x^4+28x^3z^2w+18.375x^3z^2-53.25x^3zyw\\&+13x^3zw^2-25.5x^3zw+8.75x^3z-19x^3w^2-56.5x^3w-70x^3-1.187500x^2z^3+1.375x^2z^2yw\\&-7.5x^2z^2w^2+43.25x^2z^2w+30.25x^2z^2-44.75x^2zy^2w+15x^2zyw^2-244.5x^2zyw-28x^2zw^3-177x^2zw^2\\&+300x^2zw-48.5x^2z+31.5x^2y^3w-46x^2y^2w^2+374.5x^2y^2w-12x^2yw^3+275x^2yw^2-596.5x^2yw\\&+2x^2w^4+112x^2w^3+88.5x^2w^2+128x^2w-41.5x^2+9.656250xz^4+0.6875xz^3yw-3.75xz^3w^2\\&+74.625xz^3w+28.562500xz^3-15.375xz^2y^2w+7.5xz^2yw^2-159.25xz^2yw-26xz^2w^2+126xz^2w\\&-93.75xz^2+29.75xzy^3w-25xzy^2w^2+210.25xzy^2w+36xzyw^3+169.5xzyw^2-420.25xzyw+18xzw^4-57xzw^3\\&-52.25xzw^2-266xzw+17.25xz-11xw^4-108xw^3-135xw^2+42xw-76x+14.015625z^5+0.343750z^4yw-1.875z^4w^2\\&+31.312500z^4w+33.734375z^4-7.687500z^3y^2w+3.75z^3yw^2-62.625z^3yw+33z^3w-142.468750z^3+16.875z^2y^3w\\&-16.5z^2y^2w^2+109.125z^2y^2w+22z^2yw^3+96.25z^2yw^2-204.625z^2yw+9z^2w^4-25.5z^2w^3\\&-37.625z^2w^2-147.5z^2w-2.25z^2-117.75zy^4w+117zy^3w^2-73.25zy^3w-72zy^2w^3 \\&-264.5zy^2w^2+538.25zy^2w-18zyw^4-24zyw^3+97.25zyw^2+188zyw-28zw^4+95zw^3\\&+104zw^2-51zw-164.875000z+246y^5w-238y^4w^2+120y^4w+173y^3w^3+\\&564y^3w^2-1049.5y^3w+40y^2w^4+65y^2w^3-134.5y^2w^2-440.5y^2w-34.5yw^5 \\&-37.5yw^4-285yw^3-248yw^2+243yw+17w^5+40w^4+13w^3+23w^2-42w\\
\frac{dz}{dt} & =-7x^4yw-8x^4w^2+85x^4w-16x^2y^2w^2-71x^2y^2w-28x^2yw^3-66x^2yw^2+70x^2yw+4x^2w^4+140x^2w^3\\&+84x^2w^2+278x^2w+15.468750z^5+\\&18.093750z^4-130.937500z^3+8z^2-183.75z\\ 
\frac{dy}{dt} & = 203x^5w+8x^5+30.75x^4z+3.5x^4y-59.5x^4+11.25x^3z^2-39.5x^3zy-38.25x^3z-33.5x^3y^2w+4x^3y^2\\&+38x^3yw^2-331x^3yw+59.5x^3y-268x^3w^2+722.5x^3w+38x^3+44.5x^2z^3+21x^2z^2y-1.125x^2z^2\\&+7x^2zy^2+19.75x^2zy-59x^2z+14.5x^2y^2-52x^2y-57x^2+3.90625xz^4+8.812500xz^3y+14.8125 \\&xz^3+26.625xz^2y^2+35.625xz^2y-0.25xz^2+4.25xzy^3+17.25xzy^2-2.5xzy-65.75xz\\&-243.5xy^4w+2.5xy^4+238xy^3w^2-140.5xy^3w-21.5xy^3-144xy^2w^3 \\&-549xy^2w^2+1021.5xy^2w-6xy^2-40xyw^4-48xyw^3+176.5xyw^2+438xyw-37.5xy\\&+35xw^5+36xw^4+296xw^3+315xw^2-212xw-68x+9.921875z^5-0.343750z^4yw+\\&32.375z^4y+1.875000z^4w^2-3.312500z^4w-10.546875z^4+15.75z^3y^2-4.1875z^3y\\&-96.40625z^3-44.875000z^2y^3w+31.5z^2y^3+44.5z^2y^2w^2-53.125000 \\&z^2y^2w+2.625z^2y^2-22z^2yw^3-152.25z^2yw^2+226.625z^2yw-45.875000z^2y-9z^2w^4\\&-30.5z^2w^3+39.625z^2w^2+75.5z^2w+24.25z^2-4.75zy^3-8.75zy^2-21.5 \\&zy-119.625z-246y^5w+182y^4w^2-184y^4w-15.5y^4-117y^3w^3-548y^3w^2+1027.5y^3w\\&+6.5y^3-40y^2w^4-93y^2w^3+208.5y^2w^2+527.5y^2w-16y^2+34.5yw^5+\\&105.5yw^4+299yw^3+336yw^2-169yw-93.5y-52w^5+w^4+165w^3-73w^2+74w\\ 
\frac{dw}{dt} & = 246y^6-182y^5w+168.5y^5+145y^4w^2+548y^4w-1026y^4+40y^3w^3\\&+116y^3w^2-191.5y^3w-523.5y^3-34.5y^2w^4-105.5y^2w^3-324.5y^2w^2 \\&-366y^2w+125y^2+34.5yw^4-18.5yw^3-191.5yw^2+28.5yw-128y-29w^4-74w^3+2.5w^2-41w\\
\end{align*}

\fi
\end{document}